\documentclass[a4paper,11pt]{article}
\usepackage{amsmath}
\usepackage{fancyhdr}
\usepackage{color} 
\usepackage{hyperref} 
\usepackage{graphicx}
\usepackage{enumitem}
\usepackage{mathrsfs}
\usepackage{amssymb}
\usepackage{hyperref} 
\usepackage{fullpage}
\usepackage{amsthm}

\hypersetup{colorlinks=true,linkcolor=blue,citecolor=red}
\newtheorem{thm}{Theorem}[section]                                          
\newtheorem{prp}[thm]{Proposition}                          
\newtheorem{lem}[thm]{Lemma}
\newtheorem{cly}[thm]{Corollary}

\newtheorem{dfn}[thm]{Definition}
\newtheorem{rem}[thm]{Remark}

\allowdisplaybreaks  
\numberwithin{equation}{section}
\def\Eb{\mathbb{E}}
\def\Nb{\mathbb{N}}
\def\Pb{\mathbb{P}}
\def\Rb{\mathbb{R}}
\def\Zb{\mathbb{Z}}
\def\Ac{\mathcal{A}}
\def\Fc{\mathcal{F}}
\def\Gc{\mathcal{G}}
\def\Ic{\mathcal{I}}
\def\Jc{\mathcal{J}}
\def\Lc{\mathcal{L}}
\def\Rc{\mathcal{R}}
\def\Kr{\mathbf{k}}
\def\Pr{\mathbf{P}}
\def\Tr{\mathbf{T}}
\def\Ts{\mathscr{T}}
\def\Et{\mathit{E}}
\def\Pt{\mathit{P}}
\def\th{^{\text{th}}}
\def\d{\mathrm{d}}
\def\ind{\mathbf{1}}
\def\nin{n \rightarrow \infty}
\def\IID{\emph{i.i.d.\ }}
\def\ed{\stackrel{\text{\tiny{d}}}{=}}
\def\qqquad{\qquad \quad}
\newcommand\blfootnote[1]{%
  \begingroup
  \renewcommand\thefootnote{}\footnote{#1}%
  \addtocounter{footnote}{-1}%
  \endgroup
}

\begin{document}

\title{Differentiability of the speed of biased random walks on Galton-Watson trees}

\author{Adam Bowditch\thanks{National University of Singapore, Department of Mathematics, Singapore {\tt matamb@nus.edu.sg}}
 \and Yuki Tokushige\thanks{Kyoto University, RIMS,  Kyoto 606-8502, Japan.
 {\tt tokusige@kurims.kyoto-u.ac.jp}}}

\maketitle

\begin{abstract}
We prove that the speed of a $\lambda$-biased random walk on a supercritical Galton-Watson tree is differentiable for $\lambda$ such that the walk is ballistic and obeys a central limit theorem, and give an expression of the derivative using a certain $2$-dimensional Gaussian random variable. The proof heavily uses the renewal structure of Galton-Watson trees that was introduced in \cite{LPP3}.
\end{abstract}

\blfootnote{2010 {\it Mathematics Subject Classification}.  60J80, 60K05, 60K37, 60F17.}
\blfootnote{{\it Key words and phrases}. Galton-Watson tree, biased random walks, renewal structure.}

\section{Introduction}\label{s:int}
In this paper, we investigate the speed of biased random walks on supercritical Galton-Watson trees. Specifically, we show that the speed is differentiable within a certain range of bias and obtain an expression for the derivative in terms of the covariance of a 2-dimensional Gaussian random variable. 

Random walks on GW-trees are a natural setting for studying trapping phenomena as dead-ends, caused by leaves in the trees, trap the walk. Even without leaves, the randomness in the environment slows the walk and several properties that seem obvious turn out to be non-trivial and interesting problems. These models can be used to approach related problems concerning biased random walks on percolation clusters (as studied in \cite{frha14}) and random walk in random environment (see for example \cite{S}) which experience similar phenomena. For a recent review of trapping phenomena we direct the reader to \cite{arce06}, \cite{arfr16} and \cite{foma14} which detail the history of trapping models including their motivation via spin-glasses and cover recent developments in a range of models of random walks on underlying graphs including supercritical GW-trees. 

We now briefly describe the supercritical GW-tree conditioned on survival via the Harris decomposition; for more detail see \cite{atne04, J}. Let $\{p_k\}_{k\geq 0}$ denote the offspring distribution of a GW-process $W_n$ with a single progenitor, mean $\mu >1$ and probability generating function $f$. The process $W_n$ gives rise to a random tree $\Tr_f$ where individuals are represented by vertices and edges  connect individuals with their offspring. Let $q$ denote the extinction probability of $W_n$ which is strictly less than $1$ since $\mu>1$ and non-zero only when $p_0>0$. In this case we then define 
\[g(s):=\frac{f((1-q)s+q)-q}{1-q} \qquad \text{ and } \qquad h(s):=\frac{f(qs)}{q}\]
which are generating functions of a GW-process without deaths and a subcritical GW-process respectively (cf.\ Chapter I.12 of \cite
{atne04}). An $f$-GW-tree conditioned on nonextinction $\Tr$ can be constructed by first generating a $g$-GW-tree $\Tr_g$ and then,
 to each vertex $x$ of $\Tr_g$, appending a random number of independent $h$-GW-trees (see Figure \ref{treediag}). We refer to $\Tr_g$ as the backbone of $\Tr$, the finite trees appended to $\Tr_g$ as the traps and the vertices in the first generation of the traps as the buds. 
  
\begin{figure}[!t]
\centering
 \includegraphics[scale=0.8]{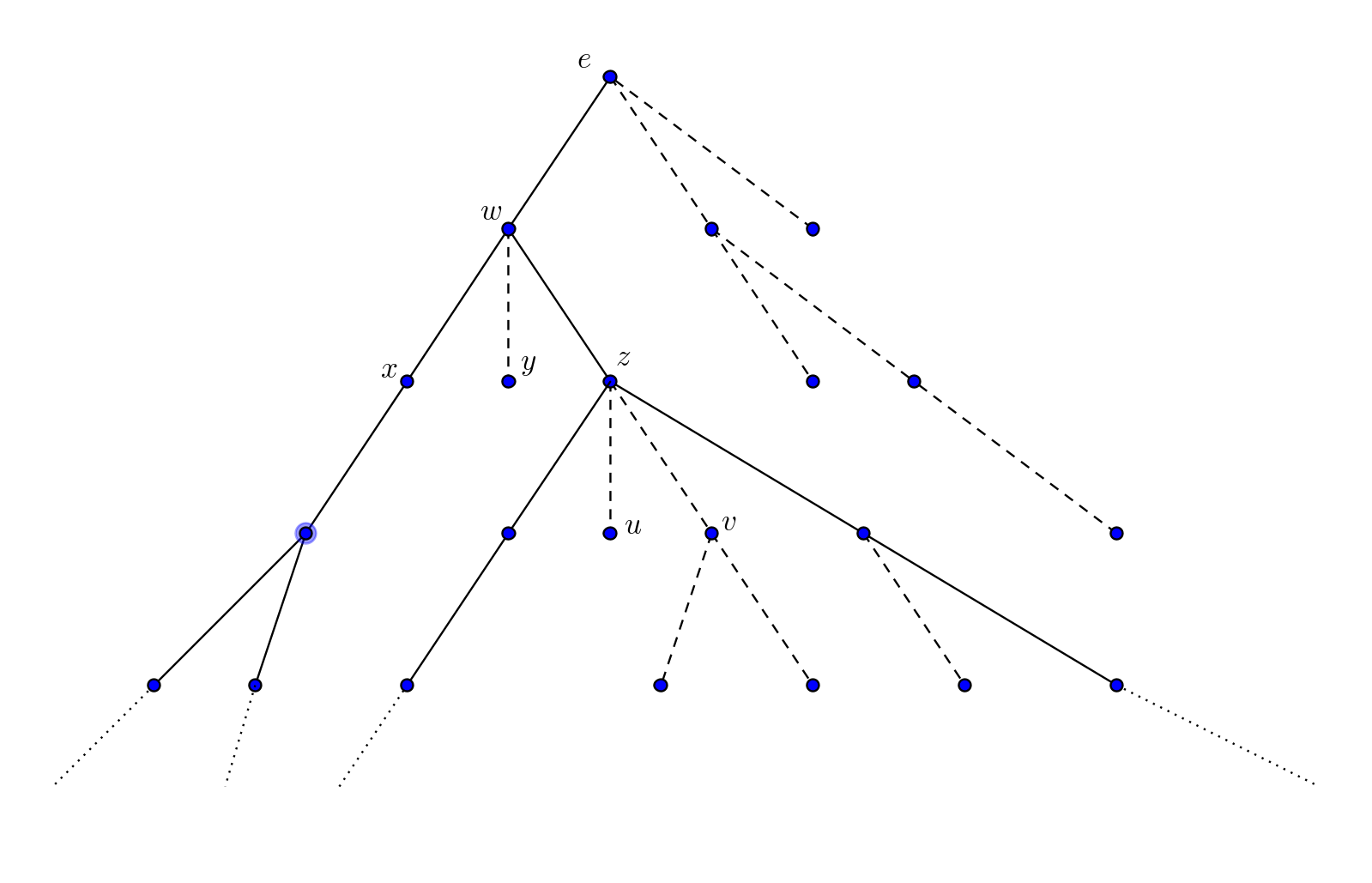} 
\caption{A sample section of a supercritical GW-tree conditioned to survive $\Tr$ with solid lines representing the backbone $\Tr_g$ and dashed lines representing the traps. Here, the root $e$ is the parent of $w$ (i.e.\ $e=\pi(w)$) which has children $x,y,z$ where $x,z$ are on the backbone and $y$ is a bud in the only trap rooted at $w$. Similarly, $u,v$ are two of the children of $z$, both of which are buds of individual traps rooted at $z$.}\label{treediag}
\end{figure}

We now introduce the biased random walk on a fixed tree $\Ts$. We denote by $e(\Ts)$ the root, which is the vertex representing
 the unique progenitor. For $x\in \Ts$, let $\pi(x)$ denote the parent of $x$ and $\nu(x)$ the number of children of $x$. A $
 \lambda$-biased random walk on $\Ts$ is a random walk $(Z_n)_{n \geq 0}$ on the vertices of $\Ts$ started from $e(\Ts)$ with
  transition probabilities 
\begin{align*}
\Pt^\Ts_\lambda(Z_{n+1}=y|Z_n=x)=A_{\lambda}(x,y):=\begin{cases} \frac{\lambda}{\lambda+\nu(x)}, & \text{if } y=\pi(x), \\  \frac{1}{\lambda
+\nu(x)}, & \text{if } x=\pi(y)\neq e(\Ts), \\ \frac{1}{\nu(x)}, & \text{if }  x=\pi(y)= e(\Ts), \\ 0, & \text{otherwise.} \\ \end{cases}
 \end{align*} 
We use $\Pt_\lambda(\cdot):=\int \Pt^\Tr_\lambda(\cdot)\Pb(\text{d}\Tr)$ for the \emph{annealed law} obtained by averaging the
 \emph{quenched law} $\Pt^\Tr_\lambda$ over the law $\Pb$ on $f$-GW-trees conditioned to survive. For $x \in \Tr$, let $d(x)$ denote the distance between $x$ and the root of the tree and write $\lambda_c:=f'(q)$ where we note that $\lambda_c=0$ when $p_0=0$. 

The behaviour of $\lambda$-biased random walks on the GW-tree $\Tr$ have been extensively studied since \cite{LPP3} showed that if $\lambda\in(\lambda_c,\mu)$ then the walk is \emph{ballistic}; that is, $d(Z_n)n^{-1}$ converges $\Pt_\lambda$-a.s.\ to a deterministic constant $\upsilon_\lambda>0$ called the {\it speed} of the walk. When $\lambda>\mu$ the walk is recurrent and $d(Z_n)n^{-1}$ converges $\Pt_\lambda$-a.s.\ to $0$. When $\lambda$ is small and $p_0>0$, the walk is transient but slowed by having to make long sequences of movements against the drift in order to escape the traps; in particular, if $\lambda\leq \lambda_c$ then the slowing affect is strong enough to cause $d(Z_n)n^{-1}$ to converge $\Pt_\lambda$-a.s.\ to $0$. This regime has been studied further in \cite{BFGH} and \cite{B2} where polynomial scaling results are shown. 

The aim of this paper is to study how the value of $\upsilon_\lambda$ depends on the parameter of bias $\lambda$; specifically, our main result is the following.
\begin{thm}\label{dif}
 Suppose that there exists $\beta>1$ such that $\sum_{k=1}^{\infty}p_k\beta^k<\infty$. Then, the function $\lambda\mapsto \upsilon_{\lambda}$ is differentiable on $(\lambda_c^{1/2},\mu)$.
 Moreover, the derivative of the speed $\upsilon'_{\lambda}$ can be expressed as the covariance of a two dimensional Gaussian random variable $(X,Y)$. Namely, we have that $\upsilon'_{\lambda}=E_{\lambda}[XY]$.
\end{thm}
We remark here that $0\leq\lambda_c\leq\lambda_c^{1/2}<1$ since $0\leq\lambda_c<1$ and note that the covariance matrix of $(X,Y)$ is given in \eqref{var}.

In the unpublished note \cite{A2}, the differentiability of the function $\lambda\mapsto \upsilon_\lambda$ was shown for $0<\lambda<1$ in the case $p_0=0$, and an expression of the derivative was given which is based on the description of invariant measures for the environment seen from the particle obtained in \cite{A1}. 

A fluctuation-dissipation theorem FDT (see \cite{dede10,ku66}) suggests that the internal fluctuations of a system at equilibrium should be related to the response of the system to an external disturbance.
In the context of a random walk, this would suggest that the fluctuations of the walk should be related to the response of imposing a drift.
A widely held conjecture is that an FDT should hold in many random walk models (e.g.\ \cite{gamapi12,lero94}); however, it has been shown in \cite{cuku93} that this is violated by several mean-field spin glass models at low temperature due to slow dynamics and aging.
This is of particular interest due to the connections between spin-glasses and models of random walks in random trapping environments.
Some progress towards proving an FDT for a random walk on a supercritical GW-tree without leaves was made in \cite{arhuolze13} where it was shown that the diffusivity is equal to the mobility (the derivative of the speed with respect to the exterior force $\alpha_\lambda=\log(\mu/\lambda)$) at the diffusive point $\lambda=\mu$.
Understanding the relation between the diffusivity and the
mobility for $\lambda$ in the ballistic regime remains open.

It has been shown in \cite{B3} and \cite{PZ} that, under the conditions of Theorem \ref{dif}, there exists a constant $\varsigma\in(0,\infty)$ such that, for $\Pb$-a.e.\ $\Tr$,
\[B_t^n:=\frac{d(Z_{nt})-nt\upsilon_\lambda}{\varsigma\sqrt{n}}\]
converges in $\Pt^{\Tr}_\lambda$-distribution to a Brownian motion. In particular, the range of bias $(\lambda_c^{1/2},\mu)$ is precisely the range in which the walk is ballistic and a central limit theorem holds. We expect that the differentiability should extend to $(\lambda_c,\mu)$ however, our proof relies heavily on second moment bounds of regeneration times which only hold in the smaller range of bias.
  
The key ingredients of the proof are the \emph{renewal structure}, the \emph{discrete Girsanov formula} and suitable moment bounds on excursion times of random walks in GW-trees. 

The renewal structure allows paths of a random walk to be decomposed into \IID components. This technique is frequently used to analyse random walks in random environments as well as various other models in probability and statistical mechanics. \cite{LPP3} constructed the renewal structure for supercritical GW-trees, which we will heavily utilise in this paper. See \cite{BFS,BGN,DGPZ} for applications of this method to the analysis of the speed of random walks in random environments. In particular, we refer to the paper \cite{BGN}, where the authors study the speed of biased random walks on a random conductance model, since our strategy resembles theirs. See \cite{M} also for a study of a similar problem in the context of random walks on word-hyperbolic groups. 

We now describe the discrete Girsanov formula which allows us to relate the walk for different values of the bias. Let $\Ts$ be a rooted infinite tree and $\bigl(\Fc_n(\Ts)\bigr)_{n\geq0}$ be the filtration on the probability space $(\tilde{\Omega}(\Ts),\Fc(\Ts),P_{\lambda}^{\Ts})$ generated by the $\lambda$-biased random walk $(Z_n)$ on $\Ts$. Then for an $\bigl(\Fc_n(\Ts)\bigr)$-stopping time $S$, an $\Fc_S(\Ts)$-measurable function $F:\tilde{\Omega}(\Ts)\rightarrow\mathbb{R}$ and $h\geq -\lambda$, we have that
\begin{align}\label{Gir}
E^{\Ts}_{\lambda+h}\left[F\bigl((Z_k)_{k\geq0}\bigr)\right]
=E^{\Ts}_{\lambda}\left[F\bigl((Z_k)_{k\geq0}\bigr)\prod_{i=1}^S\frac{A_{\lambda+h}(Z_{i-1},Z_i)}{A_{\lambda}(Z_{i-1},Z_i)}\right].
\end{align}
We remark here that regeneration times are not stopping times, thus the formula \eqref{Gir} does not apply directly
to them.
Moreover, we will mostly work with the annealed measure conditioned on {\it non-backtracking} in order to avoid bad behaviour of the first regeneration time. (See Remark \ref{rem:tau1} for details.)
The presence of the non-backtracking condition
is also an obstacle to apply the Girsanov formula. We will solve this problem using Lemma \ref{lem:replace}. 

In order to study the change in $\upsilon_\lambda$ as we vary the bias $\lambda$, we require control on the walk that is uniform in the bias. Specifically, due to the regeneration structure, it will suffice to control the variation of the walk within a single regeneration block. To this end, an important role is played by Proposition \ref{p:UniMom}, which gives a moment estimate of regeneration times that is uniform in the bias. Its proof is the main technical contribution of this paper and Sections \ref{s:uni}, \ref{s:mom} and \ref{s:reg} are entirely devoted to the fairly intricate arguments involved in it. 


The organisation of this paper is as follows; in Section \ref{s:ren}, we first introduce several basic facts on the renewal structure of GW-trees. In Section \ref{s:exp}, we will prove Theorem \ref{dif} using the formula \eqref{Gir}. We defer the more technical aspects concerning moments of regeneration times to Sections \ref{s:uni}, \ref{s:mom} and \ref{s:reg}. Specifically, in Section \ref{s:uni} we show that the uniform moment estimates for regenerations times hold for GW-trees without leaves, in Section \ref{s:mom} we prove a moment bound on the generation sizes of GW-trees and, finally, in Section \ref{s:reg} we combine these estimates to prove that the uniform moment estimates for regenerations times extends to the case with leaves.

\section{Renewal structure of Galton-Watson trees}\label{s:ren}

In this section, we introduce  regeneration times and state their moment estimates, which will be very important for this study. 
 \begin{dfn}
For a rooted tree $\Ts$ and $x\in \Ts$, define
 $P_{\lambda,x}^{\Ts}(\cdot):=P_{\lambda}^{\Ts}(\cdot|Z_0=x)$. (Thus, $P^{\Ts}_{\lambda}=P^{\Ts}_{\lambda,e(\Ts)}$.)
We will denote the expectation with respect to $P_{\lambda}$ (resp. $P_{\lambda}^{\Ts}$) by $E_{\lambda}$ (resp.\ $E_{\lambda}^{\Ts}$).
\end{dfn}
\begin{dfn}\label{d:re}
Let $(Z_n)_{n\geq0}$ be the $\lambda$-biased random walk on a rooted tree $\Ts$.
\begin{description}
\item[1] A time $n\in\Nb$ is called a regeneration time if $d(Z_n)> d(Z_k)$ for all $k<n$ and $d(Z_l)> d(Z_{n-1})$ for all $l>n$.
\item[2] For $x\in \Ts$, define the first return time $\sigma_x$ by $\sigma_x:=\inf\{n\geq1\ ;\ Z_n=x\}$.
\end{description}
\end{dfn}
\begin{rem}
Regeneration times defined above are called {\it level-regeneration times} in \cite{DGPZ}, and are different from what
 are defined in \cite{LPP3}. 
\end{rem}
\begin{dfn}Let $\Ts$ be a rooted tree. 
\begin{description}
 \item[1] For $x\in \Ts$, define $\Ts(x)$ as the subtree of $\Ts$ which consists of $x$ and its descendants. The vertex $x$ is naturally regarded as the root of $\Ts(x)$.
 \item[2] We will denote by $\Ts^*$ a new tree obtained by adding to the graph $\Ts$ an edge
 connecting $e(\Ts)$ and a new vertex $e^*(\Ts)$.
 The vertex $e^*(\Ts)$ is considered as the root of $\Ts^*$ and the parent of $e(\Ts)$.
 We often write $e$ and $e^*$ for $e(\Ts)$ and $e^*(\Ts)$ when the tree is clear from context.
\end{description}
\end{dfn}
The usefulness of renewal structure and regeneration times is that they provide a way to decompose sample paths of random walks into \IID pieces. When we deal with random walks on graphs carrying good renewal structures, approximations using regeneration times often enable us to reduce the analysis of the statistical behaviour of random walks to that of \IID random variables. 

We note that a different sequence of regeneration times (called {\it super-regeneration times}) have been introduced in \cite{BFGH} which decouple the event of a regeneration from the structure of the tree. These are particularly useful in decomposing the walk; however, this definition of regeneration times is only suitable when $\lambda<1$ because it relies on comparison with a biased random walk on $\Zb$ with this bias. 

An important property is that, by Lemma 3.3 and Proposition 3.4 of \cite{LPP3}, for $\lambda\in(0,\mu)$ there exist, $P_{\lambda}$-a.s., infinitely many regeneration times $0=:\tau_0<\tau_1<\tau_2<....$ and the sequences $\{(\tau_{i+1}-\tau_i,d(Z_{\tau_{i+1}})-d(Z_{\tau_i})\}_{i\geq0}$ are \IID random vectors under $P_{\lambda}$. A useful fact is that the law of $\left(\tau_2-\tau_1,d(Z_{\tau_2})-d(Z_{\tau_1})\right)$ under the probability measure $P_{\lambda}$ is identical to the law of $\left(\tau_1,d(Z_{\tau_1})\right)$ under the probability measure $P_{\lambda}^{\tt NB}$, where
 \begin{align*}
 P_{\lambda}^{\tt NB}(A):=
 \int\Pb({\rm d}\Tr)P_{\lambda,e(\mathbf{T})}^{\mathbf{T}^*}(A\cap\sigma_{e^*(\mathbf{T})}=\infty)\cdot
 \left(\int\Pb({\rm d}\Tr)P_{\lambda,e(\mathbf{T})}^{\mathbf{T}^*}(\sigma_{e^*(\mathbf{T})}=\infty)\right)^{-1}.
 \end{align*}
  Therefore, with respect to $P_{\lambda}^{\tt NB}$, the distribution of $\left(\tau_2-\tau_1,d(Z_{\tau_2})-d(Z_{\tau_1})\right)$ coincides with that of $\left(\tau_1,d(Z_{\tau_1})\right)$.
  We will denote by $E_{\lambda}^{\tt NB}$ the expectation with respect to $P_{\lambda}^{\tt NB}$.

The following moment estimate of regeneration times which is uniform in $\lambda$ will play an important role in this paper. We note that $\lambda_c<1$ thus the condition $a<1$ is to ensure that $\log(\lambda_c)/\log(a)>0$. Since the proof is quite technical, we postpone it to  Sections \ref{s:uni}, \ref{s:mom} and \ref{s:reg}.
 
\begin{prp}\label{p:UniMom}
Suppose $[a,b] \subset(0,\mu)$ with $a<1$ and that there exists $\beta>1$ such that $\sum_{k=1}^{\infty}p_k\beta^k<\infty$.
 Then, for any $\alpha<\log(\lambda_c)/\log(a)$ we have
\[
 \sup_{\lambda\in [a,b]}E_{\lambda}^{\tt NB}[\tau_1^{\alpha}]=\sup_{\lambda\in [a,b]}E_{\lambda}^{\tt NB}[(\tau_2-\tau_1)^{\alpha}]=
 \sup_{\lambda\in [a,b]}\Et_\lambda[(\tau_2-\tau_1)^{\alpha}]<\infty.\] 
\end{prp}
If $a>\lambda_c^{1/2}$ then $\log(\lambda_c)/\log(a)>2$ therefore we immediately have the following corollary.
\begin{cly}\label{c:2pe}
Suppose $[a,b] \subset(\lambda_c^{1/2},\mu)$ and that there exists $\beta>1$
 such that $\sum_{k=1}^{\infty}p_k\beta^k<\infty$.
 Then, for some $\varepsilon>0$ we have
\[
\sup_{\lambda\in [a,b]}E_{\lambda}^{\tt NB}[\tau_1^{2+\varepsilon}]=\sup_{\lambda\in [a,b]}E_{\lambda}^{\tt NB}[(\tau_2-\tau_1)^{2+\varepsilon}]=
\sup_{\lambda\in [a,b]}\Et_\lambda[(\tau_2-\tau_1)^{2+\varepsilon}]<\infty.\] 
\end{cly}

\begin{rem}\label{rem:tau1}
In general, $\tau_1$ does not satisfy these good moment estimates under the law $\Pt_\lambda$ which is one of the reasons that we use instead the law $\Pt_\lambda^{\tt NB}$. 
This problem arises from the number of excursions of the random walk from the root to itself until the walk
 escapes. Denote by ${\sf RE}$ the number of visits of the walk to the root, we have that $\tau_1\geq{\sf RE}\ P_{\lambda}\mathchar`-a.s.$ Moreover, ${\sf RE}$ under $P_{\lambda}^{\mathbf{T}}$ is distributed as the geometric random variable 
 whose termination probability is the quenched escape probability $P_{\lambda}^{\mathbf{T}}(\sigma_{e^*(\mathbf{T})}=\infty)$.
 
Suppose that $p_1>0$ and $\lambda>1$. 
  We denote by $R^{\lambda}(\Tr)$ the effective resistance from $e(\Tr)$ to $\infty$ of $\Tr$ where $\Tr$ is regarded as 
  the electric network corresponding to the $\lambda$-biased random walk. 
 Let us consider the event where every individual up to $n\th$ generation 
 gives birth to only one child, which occurs with probability $p_1^n$. 
 On this event, $R^{\lambda}(\Tr)$ is of order $\lambda^n$. This observation implies that
 $$\Pb(R^{\lambda}(\Tr)>n)\geq cn^{\log p_1/\log\lambda}$$
 for some constant $c>0$. (Note that $\log p_1<0$ since $0<p_1<1$.)
 In the light of a well-known fact in the theory of electric networks and reversible random walks (see Theorem 2.11 of \cite{Ba17} for instance), the above estimate implies that ${\sf RE}$ (therefore $\tau_1$ also) does not even satisfy a finite first moment in general.
 \end{rem}
 
We also need the following estimates in what follows. 
Their proofs will also be given in Sections \ref{s:uni}, \ref{s:mom} and \ref{s:reg}.
\begin{prp}\label{est:sigma}
Suppose that there exists $\beta>1$ such that $\sum_{k=1}^{\infty}p_k\beta^k<\infty$. Then for any $\lambda\in(\lambda_c,\mu)$, we have that  
\[\Eb\left[E_{\lambda}^{\mathbf{T}^*}\left[\sigma_{e^*(\mathbf{T})}\ind_{\{\sigma_{e^*(\mathbf{T})}<\infty\}}\right]\right]<\infty.\]
\end{prp}
\begin{prp}\label{p:esc-conti}
The function
\[\lambda\mapsto\Eb\left[P_{\lambda}^{\Tr^*}\left(\sigma_{e^*(\Tr)}=\infty\right)\right]\]
is continuous on $(\lambda_c,\mu)$.
\end{prp}

\section{Expressions of derivatives of the speed}\label{s:exp}
In this section we prove Theorem \ref{dif} assuming Proposition \ref{p:UniMom}.
The following result gives the finite approximation of the derivative.
 Notice that we do not use expectations $E_{\lambda}[d(Z_n)]$ and $E_{\lambda+h}[d(Z_n)]$, but $E_{\lambda}^{\tt NB}[d(Z_n)]$ and $E_{\lambda+h}^{\tt NB}[d(Z_n)]$, to approximate the derivative, which is
 a non-negligible difference from the approach of \cite{BGN}.
 This is necessary because of the lack of good moment estimates of the first regeneration time $\tau_1$, which
 arises from a special role played by the root. See Remark \ref{rem:tau1} for details.
 
\begin{prp}\label{appro} Suppose $\lambda\in(\lambda_c^{1/2},\mu)$ and that there exists $\beta>1$ such that $\sum_{k=1}^{\infty}p_k\beta^k<\infty$. Let $h$ tend to $0$ and $n$ tend to $\infty$ in such a way that $h^2n$ tends to $1$ (i.e.\ $hn\sim n^{1/2}$). Then
\begin{align*}
\frac{\upsilon_{\lambda+h}-\upsilon_\lambda}{h}-\frac{E_{\lambda+h}^{\tt NB}[d(Z_n)]-E_{\lambda}^{\tt NB}[d(Z_n)]}{hn}
\end{align*}
 tends to $0$.
\end{prp}
\begin{proof}
 Define $\eta_n:=\inf\{k:\tau_k\geq n\}$ for $n\in\mathbb{N}$, then $\eta_n$ is a stopping time with respect to the filtration
 generated by random variables $\tau_1$, and $\{\tau_{i+1}-\tau_i\}_{i\geq1}$.
 By the definition of $\eta_n$, we have
 \begin{align}\label{e:Sdw}
 n\leq\tau_{\eta_n}\leq n+\max_{0\leq i\leq n}(\tau_{i+1}-\tau_i),\ P_{\lambda}^{\tt NB}\mathchar`-a.s.
 \end{align}
Combining this with Wald's identity we then have
 \begin{align}\label{a}
 n\leq E_{\lambda}^{\tt NB}[\tau_{\eta_n}]=E_{\lambda}^{\tt NB}[\eta_n]E_{\lambda}^{\tt NB}[\tau_1]\leq n
 +E_{\lambda}^{\tt NB}\left[\max_{0\leq i\leq n}(\tau_{i+1}-\tau_i)\right].
 \end{align} 
 
Using that $0\leq |d(Z_n)-d(Z_{\tau_{\eta_n}})|\leq \tau_{\eta_n}-n$, \eqref{e:Sdw} then implies that
 \begin{align*}
 |E_{\lambda}^{\tt NB}[d(Z_n)]-E_{\lambda}^{\tt NB}[d(Z_{\tau_{\eta_n}})]|
 \leq E_{\lambda}^{\tt NB}\left[\max_{0\leq i\leq n}(\tau_{i+1}-\tau_i)\right].
 \end{align*}
 Wald's identity then gives 
\begin{align*}
 E_{\lambda}^{\tt NB}[d(Z_{\tau_{\eta_n}})]=E_{\lambda}^{\tt NB}[\eta_n]E_{\lambda}^{\tt NB}[d(Z_{\tau_1})].
\end{align*}
 Hence, we get
 \begin{align}\label{aa}
 \Bigl|E^{\tt NB}_{\lambda}[d(Z_n)]-n\upsilon_\lambda \Bigr|
 &\leq \left|E^{\tt NB}_{\lambda}[d(Z_n)]-E^{\tt NB}_{\lambda}[d(Z_{\tau_{\eta_n}})] \right|+
 \left|E^{\tt NB}_{\lambda}[d(Z_{\tau_{\eta_n}})]-n\upsilon_\lambda \right| \nonumber\\
& \leq E_{\lambda}^{\tt NB}\left[\max_{0\leq i\leq n}(\tau_{i+1}-\tau_i)\right]+\left|E_{\lambda}^{\tt NB}[\eta_n]
  E_{\lambda}^{\tt NB}[d(Z_{\tau_1})]-n\upsilon_\lambda \right|
 \end{align}
 
 By (6.4) in \cite{LPP3} we have that
 \begin{align}\label{speed}
 \upsilon_\lambda=\frac{E_{\lambda}[d(Z_{\tau_2})-d(Z_{\tau_1})]}{E_{\lambda}[\tau_2-\tau_1]}
=\dfrac{E_{\lambda}^{\tt NB}[d(Z_{\tau_1})]}{E_{\lambda}^{\tt NB}[\tau_1]},
 \end{align}
therefore, using \eqref{a} and that $\upsilon_\lambda\leq 1$, we have
 \begin{align}\label{aaa}
 \left|E_{\lambda}^{\tt NB}[\eta_n]
  E_{\lambda}^{\tt NB}[d(Z_{\tau_1})]-n\upsilon_\lambda \right|\leq\left|E_{\lambda}^{\tt NB}[\eta_n]E_{\lambda}^{\tt NB}[\tau_1]-n\right|\leq E_{\lambda}^{\tt NB}\left[\max_{0\leq i\leq n}(\tau_{i+1}-\tau_i)\right].
 \end{align}
 By combining \eqref{aa} and \eqref{aaa},
 we get
\begin{align*}
 \left|E_{\lambda}^{\tt NB}[d(Z_n)]-n\upsilon_\lambda \right|\leq 2E_{\lambda}^{\tt NB}\left[\max_{0\leq i\leq n}(\tau_{i+1}-\tau_i)\right]
\end{align*}
 In order to complete the proof of Proposition \ref{appro},
 it suffices to show that there exist constants $0<t_{\lambda}<\min\{\mu-\lambda,\lambda-\lambda_c\}$, $0<\kappa<1/2$ and $c_{\lambda}>0$ such that 
 \begin{align}\label{max}
 E_{\lambda'}^{\tt NB}\left[\max_{0\leq i\leq n}(\tau_{i+1}-\tau_i)\right]\leq c_{\lambda} n^{\kappa}
 \end{align}
 for any $\lambda'\in(\lambda-t_{\lambda},\lambda+t_{\lambda})$. The estimate \eqref{max} can be proved as follows:
 for $\kappa>0$, we have 
 \begin{align*}
 &E_{\lambda'}^{\tt NB}\left[\max_{0\leq i\leq n}(\tau_{i+1}-\tau_i)\right]\leq
  n^{\kappa}+\sum_{k\geq[n^{\kappa}]}P_{\lambda'}^{\tt NB}\left(\max_{0\leq i\leq n}(\tau_{i+1}-\tau_i)\geq k\right).
 \end{align*}
 By Corollary \ref{c:2pe}, there exists a constant $c_{\lambda}>0$
 such that for $\lambda'\in(\lambda-t_{\lambda},\lambda+t_{\lambda})$
 and sufficiently large $k$, we have 
 \begin{align*}
 P_{\lambda'}^{\tt NB}\left(\max_{0\leq i\leq n}(\tau_{i+1}-\tau_i)\geq k\right)&=
 1-\{1-P_{\lambda'}\left(\tau_{2}-\tau_1\geq k\right)\}^n\\
 &\leq 1-(1-c_{\lambda}k^{-(2+\varepsilon)})^n\leq 2c_{\lambda}nk^{-(2+\varepsilon)}.
 \end{align*}
 Thus, for sufficiently large $n$, we get
 \begin{align*}
 E_{\lambda'}^{\tt NB}\left[\max_{0\leq i\leq n}(\tau_{i+1}-\tau_i)\right]\leq
 n^{\kappa}+4c_{\lambda}n^{-\kappa(1+\varepsilon)+1}.
 \end{align*}
 Since $\max\{\kappa,-\kappa(1+\varepsilon)+1\}\leq\frac{1}{2+\varepsilon}$,
 we obtain the estimate \eqref{max}. 
 Therefore, we have shown that
 \begin{align*}
 \frac{E_{\lambda+h}^{\tt NB}[d(Z_n)-n\upsilon_{\lambda+h}]-E_{\lambda}^{\tt NB}[d(Z_n)-n\upsilon_\lambda]}{hn}
 \end{align*}
 tends to $0$ when $h$ tends to $0$ and $n$ tends to $\infty$
 in such a way that $h^2n$ tends to $1$. 
 \end{proof}
 
By Proposition \ref{appro}, in order to show the differentiability of the function $\lambda\mapsto \upsilon_\lambda$, it suffices
 to prove the existence of the limit 
 \begin{align*}
 \lim_{h,n}\frac{E_{\lambda+h}^{\tt NB}[d(Z_n)]-E_{\lambda}^{\tt NB}[d(Z_n)]}{hn},
 \end{align*}
 where $h$ tends to $0$ and $n$ tends to $\infty$ in such a way that $h^2n$ tends to $1$.
We will need the following estimates.
 \begin{lem}\label{ui}
 Suppose that $\lambda\in(\lambda_c^{1/2},\mu)$
 and that there exists $\beta>1$
 such that $\sum_{k=1}^{\infty}p_k\beta^k<\infty$.
 Then, we have 
 \begin{align}\label{est:ui1}
 \sup_n\frac{1}{n}E_{\lambda}^{\tt NB}[(d(Z_n)-n\upsilon_\lambda)^2]<\infty
 \end{align}
and
\begin{align}\label{est:ui2}
 \sup_n\frac{1}{n}\Eb\left[E^{\mathbf{T}^*}_{\lambda}\left[(d(Z_n)-n\upsilon_\lambda)^{2}\ind_{\{\sigma_{e^*(\mathbf{T})}>n\}}\right]\right]<\infty.
\end{align}
 \end{lem}
 \begin{proof}
We first prove \eqref{est:ui1}.
Using \eqref{e:Sdw} and the arguments of Proposition \ref{appro} we have that
 \begin{align*}
 \left|d(Z_n)-n\upsilon_\lambda\right|\leq \left|d(Z_{\tau_{\eta_n}})-n\upsilon_\lambda\right|+\max_{0\leq i\leq n}(\tau_{i+1}-\tau_i)\ \ P^{\tt NB}_{\lambda}\mathchar`-a.s.,
 \end{align*}
 and, $P^{\tt NB}_{\lambda}$-a.s.,
\begin{align*}
d(Z_{\tau_{\eta_n}})-n\upsilon_\lambda
=\sum_{i=0}^{\eta_n-1}\left(d(Z_{\tau_{i+1}})-d(Z_{\tau_i})-E_{\lambda}^{\tt NB}[d(Z_{\tau_1})]\right)+\bigl(\eta_n\cdot E_{\lambda}^{\tt NB}[d(Z_{\tau_1})]-n\upsilon_\lambda\bigr) 
\end{align*}
 Hence, we get
\begin{align*}
 & E^{\tt NB}_{\lambda}[(d(Z_n)-n\upsilon_\lambda)^2]\\
&\leq 4\Biggl\{E_{\lambda}^{\tt NB}\left[\left(\max_{0\leq i\leq n}(\tau_{i+1}-\tau_i)\right)^2\right]
+E_{\lambda}^{\tt NB}\left[\left(\eta_n\cdot E^{\tt NB}_{\lambda}[d(Z_{\tau_1})]-n\upsilon_\lambda\right)^2\right]\\
&\ \ \ \ \ \ \ \ \ \ \ \ \ \ \ \ \ \ \ \ \ \ \ \ \ \ \ \ \ \ \ \ +E_{\lambda}^{\tt NB}\left[\left\{\sum_{i=0}^{\eta_n-1}\left(d(Z_{\tau_{i+1}})-d(Z_{\tau_i})-E_{\lambda}^{\tt NB}[d(Z_{\tau_1})]\right)\right\}^2\right]\Biggr\}.
\end{align*}
Recall that $d(Z_{\tau_{i+1}})-d(Z_{\tau_i})$ are \IID then Wald's second identity implies
\begin{align*}
& E_{\lambda}^{\tt NB}\left[\left(\sum_{i=0}^{\eta_n-1}\left(d(Z_{\tau_{i+1}})-d(Z_{\tau_i})-E_{\lambda}^{\tt NB}[d(Z_{\tau_1})]\right)\right)^2\right]\\
&=E_{\lambda}^{\tt NB}\left[\left(d(Z_{\tau_{2}})-d(Z_{\tau_1})-E_{\lambda}^{\tt NB}[d(Z_{\tau_1})]\right)^2\right]E_{\lambda}[\eta_n]
\end{align*}
 thus by Corollary \ref{c:2pe} and the estimate \eqref{a} along with \eqref{max}, we have
 \begin{align*}
 \sup_{n}\frac{1}{n}E_{\lambda}^{\tt NB}\left[\left(\sum_{i=0}^{\eta_n-1}\left(d(Z_{\tau_{i+1}})-d(Z_{\tau_i})-E_{\lambda}^{\tt NB}[d(Z_{\tau_1})]\right)\right)^2\right]<\infty.
 \end{align*}
It is not difficult to see that Corollary \ref{c:2pe}
 implies $n^{-1}E_{\lambda}^{\tt NB}[(\max_{0\leq i\leq n}(\tau_{i+1}-\tau_i))^2]$ is also bounded in $n$. 
 Hence, we get the conclusion if we show 
 \begin{align*}
 \sup_{n\geq1}\frac{1}{n}E_{\lambda}^{\tt NB}\left[\Bigl(\eta_n\cdot E_{\lambda}^{\tt NB}[d(Z_{\tau_1})]-n\upsilon_\lambda\Bigr)^2\right]<\infty.
 \end{align*}
  It is shown in Chapter 4 of \cite{C} that 
 \begin{align*}
 E_{\lambda}^{\tt NB}[\eta_n^2]=E_{\lambda}^{\tt NB}[\eta_n]^2+O(n)=\frac{n^2}{E_{\lambda}^{\tt NB}[\tau_1]^2}+O(n).
 \end{align*}
 By using the formula \eqref{speed} 
 and the estimate \eqref{a}, we get
\begin{align*}
&E_{\lambda}^{\tt NB}\left[\Bigl(\eta_n\cdot E_{\lambda}^{\tt NB}[d(Z_{\tau_1})]-n\upsilon_\lambda\Bigr)^2\right] \\
&=E_{\lambda}^{\tt NB}[\eta_n]^2E_{\lambda}^{\tt NB}[d(Z_{\tau_1})]^2-\frac{E_{\lambda}^{\tt NB}[d(Z_{\tau_1})]^2}
 {E_{\lambda}^{\tt NB}[\tau_1]^2}\cdot n^2+O(n) 
\end{align*}
 which is at most order $n$ therefore this implies \eqref{est:ui1}.
 
 We next prove \eqref{est:ui2}. In order to deduce \eqref{est:ui2} from \eqref{est:ui1},
 it suffice to show that
 \begin{align*}
\sup_{n}\frac{1}{n}\Eb\left[E^{\mathbf{T}^*}_{\lambda}\left[(d(Z_n)-n\upsilon_\lambda)^{2}\ind_{\{n<\sigma_{e^*(\mathbf{T})}<\infty\}}\right]\right]<\infty.
 \end{align*}
 This immediately follows from Proposition \ref{est:sigma} and an obvious bound
  $|d(Z_n)-n\upsilon_\lambda|\leq (1+\upsilon_{\lambda})n.$
   \end{proof}

 Lemma \ref{ui} implies uniform integrability of the sequence $\{(d(Z_n)-n\upsilon_\lambda)/\sqrt{n}\}_{n\geq1}$
  under the conditioned annealed measure $P_{\lambda}^{\tt NB}$
  when $\lambda\in(\lambda_c^{1/2},\mu)$ and the offspring distribution $\{p_k\}_{k\geq0}$ has finite exponential moment.
 On the other hand, by Corollary \ref{c:2pe} and a standard argument in the renewal theory, we get that 
  the sequence $\{(d(Z_n)-n\upsilon_\lambda)/\sqrt{n}\}_{n\geq1}$ satisfies the annealed CLT under the same assumptions.
  Hence, we have
 \begin{align*}
 \lim_{n\to\infty}\frac{E_{\lambda}^{\tt NB}[d(Z_n)-n\upsilon_\lambda]}{\sqrt{n}}=0,
 \end{align*}
 for any $\lambda\in(\lambda_c^{1/2},\mu)$. Thus, in order to prove the differentiability of the speed,
 we only need to show the existence of the limit
 \begin{align}\label{aim}
 \frac{1}{hn}E_{\lambda+h}^{\tt NB}[d(Z_n)-n\upsilon_\lambda],
 \end{align}
 for any sequence $h$ and $n$ such that $h\rightarrow0$ and $n\rightarrow\infty$
 in such a way that $h^2n\rightarrow1$. To do so, we wish to relate the measures $E_{\lambda}^{\tt NB}$ and
 $E_{\lambda+h}^{\tt NB}$ by using the Girsanov formula \eqref{Gir}.
However, the formula \eqref{Gir} does not directly apply to $E_{\lambda+h}^{\tt NB}[d(Z_n)-n\upsilon_\lambda]$
because of the presence of the non-backtracking condition $\{\sigma_{e^*(\mathbf{T})}=\infty\}$.
In order to overcome this problem, we need the following lemma.
\begin{lem}\label{lem:replace}
Suppose that $\lambda\in(\lambda_c^{1/2},\mu)$
 and that there exists $\beta>1$
 such that $\sum_{k=1}^{\infty}p_k\beta^k<\infty$.
 Then, we have that
\begin{align}\label{eq:replace}
 &\frac{1}{\sqrt{n}}\Biggl(
\Eb\left[E^{\mathbf{T}^*}_{\lambda+h}\left[\left(d(Z_n)-n\upsilon_\lambda\right)\ind_{\{\sigma_{e^*(\mathbf{T})}=\infty\}}\right]\right]\notag\\
&\ \ \ \ \ \ \ \ \ \ \   -\Eb\left[E^{\mathbf{T}^*}_{\lambda}\left[\left(d(Z_n)-n\upsilon_\lambda\right)\prod_{i=1}^{n}\dfrac{A_{\lambda+h}(Z_{i-1},Z_i)}{A_{\lambda}(Z_{i-1},Z_i)}\ind_{\{\sigma_{e^*(\mathbf{T})}=\infty\}}\right]\right]\Biggr)
\end{align}
 converges to $0$ as $h$ tends to $0$ and $n$ tends to $\infty$ in such a way that $h^2n\to1$.
\end{lem}
Since the proof of Lemma \ref{lem:replace} requires a careful analysis of the Girsanov weight, we will defer it until the end of the next subsection.

\subsection{The discrete Girsanov formula}
 In this subsection, we will analyse the Girsanov weight $\prod_{i=1}^n\frac{A_{\lambda+h}(Z_{i-1},Z_i)}{A_{\lambda}(Z_{i-1},Z_i)}$.
\ By the Taylor expansion, there exists $s=s(x,y)\in[0,1]$ such that 
 \begin{align*}
 \log\frac{A_{\lambda+h}(x,y)}{A_{\lambda}(x,y)}
=hB_{\lambda}(x,y)&+\frac{h^2}{2}C_{\lambda}(x,y)+\frac{h^3}{6}D_{\lambda+sh}(x,y),
 \end{align*}
 where
 \begin{align*}
 B_{\lambda}(x,y)=\frac{\d}{\d\lambda}\log A_{\lambda}(x,y)&=\begin{cases}
 0 &{\rm when}\ x=e,\\
  \frac{1}{\lambda}-\frac{1}{\lambda+\nu(x)} &{\rm when}\ y=\pi(x),\\
   -\frac{1}{\lambda+\nu(x)}          &{\rm when}\ x=\pi(y),
 \end{cases}\\
 C_{\lambda}(x,y)=\frac{\d}{\d\lambda}B_{\lambda}(x,y)&=\begin{cases}
 0 &{\rm when}\ x=e,\\
  -\frac{1}{\lambda^2}+\frac{1}{(\lambda+\nu(x))^2} &{\rm when}\ y=\pi(x),\\
   \frac{1}{(\lambda+\nu(x))^2}          &{\rm when}\ x=\pi(y),
 \end{cases}
 \intertext{and}
 D_{\lambda}(x,y)=\frac{\d}{\d\lambda}C_{\lambda}(x,y)&=\begin{cases}
 0 &{\rm when}\ x=e,\\
  \frac{2}{\lambda^3}-\frac{2}{(\lambda+\nu(x))^3} &{\rm when}\ y=\pi(x),\\
   -\frac{2}{(\lambda+\nu(x))^3}          &{\rm when}\ x=\pi(y).
   \end{cases}
 \end{align*}
 By using these expressions, we have
 \begin{align}\label{AAA}
\prod_{i=1}^{n}\frac{A_{\lambda+h}(Z_{i-1},Z_i)}{A_{\lambda}(Z_{i-1},Z_i)}=\exp(hP_n-h^2Q_n+R_{n,h})\ \ P_{\lambda}\mathchar`-a.s.,
 \end{align}
 where
 \begin{align*}
 P_n&:=\sum_{j=0}^{n-1}B_{\lambda}(Z_j,Z_{j+1}),\\
 Q_n&:=\sum_{j=0}^{n-1}\frac{1}{2}B^2_{\lambda}(Z_j,Z_{j+1}),\\ 
 R_{n,h}&:=\sum_{j=0}^{n-1}\left\{h^2\left(\frac{1}{2}B_{\lambda}^2(Z_j,Z_{j+1})+\frac{1}{2}C_{\lambda}(Z_j,Z_{j+1})\right)
+\frac{h^3}{6}D_{\lambda+sh}(Z_j,Z_{j+1})\right\}.
 \end{align*}
 Since 
 \begin{align}\label{Ubound}
 |B_{\lambda}(x,y)|\leq\frac{1}{\lambda}+1,\ 
 |C_{\lambda}(x,y)|\leq\frac{1}{\lambda^2}+1,\ |D_{\lambda}(x,y)|\leq\frac{2}{\lambda^3}+2,
 \end{align}
  we get
 \begin{align*}
 1&=\sum_{y}A_{\lambda+h}(x,y)=\sum_{y}A_{\lambda}(x,y)\exp\Bigl(hB_{\lambda}(x,y)+\frac{h^2}{2}C_{\lambda}(x,y)
+\frac{h^3}{6}D_{\lambda+sh}(x,y)\Bigr)\\
&=\sum_{y}A_{\lambda}(x,y)\Bigl(1+hB_{\lambda}(x,y)+\frac{h^2}{2}B^2_{\lambda}(x,y)+\frac{h^2}{2}C_{\lambda}(x,y)+O(h^3)\Bigr).
 \end{align*}
 This implies that for any $x\in\Tr$,
 \begin{align}\label{0}
 \sum_{y}A_{\lambda}(x,y)B_{\lambda}(x,y)&=0,\ \ P_{\lambda}\mathchar`-a.s.,\\
 \sum_{y}A_{\lambda}(x,y)\Bigl(B^2_{\lambda}(x,y)+C_{\lambda}(x,y)\Bigr)&=0,\ \ P_{\lambda}\mathchar`-a.s. \notag
 \end{align}
 By using the Markov property and the equality \eqref{0}, we obtain
 \begin{align*}
 E_{\lambda}^{\mathbf{T}}[B_{\lambda}(Z_j,Z_{j+1})\ |\ Z_j]=\sum_{y}B(Z_j,y)A_{\lambda}(Z_j,y)=0,\ \ P^{\mathbf{T}}_{\lambda}\mathchar`-a.s.
 \end{align*}
 This implies 
 \begin{align}\label{B}
 E^{\mathbf{T}}_{\lambda}[B_{\lambda}(Z_j,Z_{j+1})]=E_{\lambda}[B_{\lambda}(Z_j,Z_{j+1})]=0,\ \ P_{\lambda}\mathchar`-a.s.
 \end{align}
 Similarly,
 we have 
 \begin{align}\label{BC}
 E^{\mathbf{T}}_{\lambda}[B^2_{\lambda}(Z_j,Z_{j+1})+C_{\lambda}(Z_j,Z_{j+1})]=0,\ \ P_{\lambda}\mathchar`-a.s.
 \end{align}
We now let $h$ tend to $0$ and $n$ tend to $\infty$ in such a way that $h^2n$ tends to $1$. We show that the limits of $hP_n$ and $h^2Q_n$ are described by a CLT and a LLN respectively and the limit of $R_{n,h}$ is negligible.

 \begin{enumerate}[align=left, leftmargin=0pt, labelindent=\parindent, listparindent=\parindent, labelwidth=0pt, itemindent=!]
\item[{\bf 1) The CLT for $P_n$:}]   
By the renewal structure of GW-trees, we know that the collection $\{\sum_{j=\tau_i}^{\tau_{i+1}-1}B_{\lambda}(Z_j,Z_{j+1})\}_{i\geq1}$ are \IID random variables under $P_{\lambda}$, and are distributed as $\sum_{i=0}^{\tau_1-1}B_{\lambda}(Z_j,Z_{j+1})$ under $P_{\lambda}^{\tt NB}$. 
Recall that by \eqref{B} we have that 
\begin{align}\label{eq:pn}
E_{\lambda}[P_n]=0.
\end{align} 
On the other hand, noticing that $\tau_1$ is finite a.s. by the renewal structure we have the following LLN:
\begin{align}\label{e:lln}
\lim_{n\to\infty}n^{-1}P_n=\dfrac{E_{\lambda}\left[\sum_{j=\tau_1}^{\tau_2-1}B_{\lambda}(B_j,B_{j+1})\right]}{E_{\lambda}[\tau_2-\tau_1]}=\dfrac{E_{\lambda}^{\tt NB}\left[\sum_{j=0}^{\tau_1-1}B_{\lambda}(B_j,B_{j+1})\right]}{E^{\tt NB}_{\lambda}[\tau_1]}\ \ P_{\lambda}\mathchar`-a.s.
\end{align}
Noticing that $n^{-1}P_n\leq 1+\lambda^{-1}$ by \eqref{Ubound}, the dominated convergence theorem implies that the same convergence as \eqref{e:lln} holds in $L^1(P_{\lambda})$. This fact together with \eqref{eq:pn} implies that
\begin{align}\label{e:BTT}
E_{\lambda}\left[\sum_{j=\tau_1}^{\tau_2-1}B_{\lambda}(Z_j,Z_{j+1})\right]=E_{\lambda}^{\tt NB}\left[\sum_{j=0}^{\tau_1-1}B_{\lambda}(Z_j,Z_{j+1})\right]=0.
\end{align}
By \eqref{Ubound} and Corollary \ref{c:2pe} we also have that 
\begin{align*}
E_{\lambda}\left[\left(\sum_{j=\tau_1}^{\tau_2-1}B_{\lambda}(Z_j,Z_{j+1})\right)^2\right]
&\leq E_{\lambda}^{\tt NB}\left[\left(\sum_{j=0}^{\tau_1-1}B_{\lambda}(Z_j,Z_{j+1})\right)^2\right]\\
&\leq \left(1+\frac{1}{\lambda}\right)^2E_{\lambda}^{\tt NB}\left[\tau_1^2\right]<\infty.
\end{align*}
Moreover, we have that
\[n^{-1/2}\left(\sum_{j=0}^{\tau_{\eta_n}-1}B_{\lambda}(Z_j,Z_{j+1})-P_n\right)\leq n^{-1/2}\left(1+\frac{1}{\lambda}\right)\max_{0\leq i\leq n}(\tau_{i+1}-\tau_{i}),\]
which converges to $0$ in probability by \eqref{max}.
It therefore follows that $n^{-1/2}P_n$ converges in distribution to a centred Gaussian.



\item[{\bf 2) The LLN for $Q_n$:}] 
Recalling that $\tau_1$ is $P_{\lambda}$-a.s.\ finite and from \eqref{Ubound} that $B_{\lambda}^2(Z_j,Z_{j+1})$ is bounded above, by the law of large numbers we know that, $P_{\lambda}$-a.s.

\begin{align*}
\lim_{n\to\infty}\frac{1}{n}Q_n
&=\frac{1}{2E_{\lambda}[\tau_2-\tau_1]}E_{\lambda}\left[\sum_{j=\tau_1}^{\tau_2-1}B_{\lambda}^2(Z_j,Z_{j+1})\right] \\
&=\frac{1}{2E_{\lambda}^{\tt NB}[\tau_1]}E^{\tt NB}_{\lambda}\left[\sum_{j=0}^{\tau_1-1}B_{\lambda}^2(Z_j,Z_{j+1})\right].
\end{align*}

\item[{\bf 3) The estimate for $R_{n,h}$:}] 
For some constant $c<\infty$, we have $n\leq ch^{-2}$. Using this and \eqref{Ubound}, we have
\begin{align*}
\left|\sum_{j=0}^{n-1}\frac{h^3}{6}D_{\lambda+sh}(Z_j,Z_{j+1})\right|\leq \frac{ch}{3}\left(\frac{1}{\lambda^3}+1\right).
\end{align*}
By \eqref{BC} we have that
\begin{align*}
E_{\lambda}^{\Tr}&\left[\sum_{j=0}^{n-1}\left(B_{\lambda}^2(Z_j,Z_{j+1})+C_{\lambda}(Z_j,Z_{j+1})\right)\right]=0,\ P_{\lambda}\mathchar`-a.s. 
\end{align*}
hence, by using the similar argument to the above one, we see that
\begin{align*}
\lim_{n\to\infty}R_{n,h}=0,\ P_{\lambda}\mathchar`-a.s. 
\end{align*}
 Note also that $R_{n,h}$ satisfies the following uniform estimate for $h$ sufficiently small. 
 \begin{align}\label{R_n}
 |R_{n,h}|&\leq h^2n\left(\frac{1}{\lambda}+1+\frac{1}{2}\left(\frac{1}{\lambda^2}+1\right)\right)+\frac{ch}{3}\left(\frac{1}{\lambda^3}+1\right) \nonumber\\
 &\leq 2c\left(\frac{1}{\lambda}+1+\frac{1}{2}\left(\frac{1}{\lambda^2}+1\right)\right)+\frac{1}{\lambda^3}+1.
 \end{align}
 
 \item[{\bf 4) The joint CLT for}] 
 $\Bigl(n^{-1/2}(d(Z_n)-n\upsilon_\lambda),n^{-1/2}P_n\Bigr)_{n\geq1}${\bf :}
 We have given a proof of the annealed CLT for the sequences of random variables 
 $\{n^{-1/2}(d(Z_n)-n\upsilon_\lambda)\}_{n\geq1}$ and $\{n^{-1/2}P_n\}_{n\geq1}$,
 but in what follows, we need the joint CLT for the sequence of random vectors
 $\bigl(n^{-1/2}(d(Z_n)-n\upsilon_\lambda),n^{-1/2}P_n\bigr)$. 
 Note that for any $\lambda\in(\lambda_c^{1/2},\mu)$, 
 \begin{align*}
 \left(d(Z_{\tau_{l+1}})-d(Z_{\tau_l})-\upsilon_\lambda(\tau_{l+1}-\tau_l),\sum_{j=\tau_l}^{\tau_{l+1}-1}B_{\lambda}(Z_j,Z_{j+1})\right)_{l\geq1}
 \end{align*}
 are \IID $\mathbb{R}^2$-valued random variables under $P_{\lambda}$.
  \end{enumerate}
This fact together with the moment estimate of regeneration times immediately implies the following result. Note that for $\sigma_{10}(\lambda)$, we use $E_{\lambda}\left[\sum_{j=\tau_1}^{\tau_2-1}B_{\lambda}(Z_j,Z_{j+1})\right]=0$ from \eqref{e:BTT} and that $\sigma_{00}(\lambda)$ coincides with the diffusion constant in the central limit theorems proved in \cite{B3}.
 \begin{prp}\label{joint}
  Suppose $\lambda\in(\lambda_c^{1/2},\mu)$
 and that there exists $\beta>1$
 such that $\sum_{k=1}^{\infty}p_k\beta^k<\infty$.
 Then, the sequence $\bigl\{(n^{-1/2}(d(Z_n)-n\upsilon_\lambda),n^{-1/2}P_n)\bigr\}_{n\geq1}$
 under $P_{\lambda}$ converges weakly to
 the two dimensional Gaussian random variable $(X,Y)$ with the covariance matrix
 $\Sigma_{\lambda}:=(\sigma_{ij}(\lambda))_{0\leq i,j\leq1}$ given by
 \begin{align}\label{var}
 \sigma_{00}(\lambda)&:=\frac{1}{E_{\lambda}[\tau_2-\tau_1]}
 E_{\lambda}\left[\Bigl(\bigl(d(Z_{\tau_2})-d(Z_{\tau_1})\bigr)-E_{\lambda}[d(Z_{\tau_2})-d(Z_{\tau_1})]\Bigr)^2\right],\nonumber\\
 \sigma_{11}(\lambda)&:=\frac{1}{E_{\lambda}[\tau_2-\tau_1]}E_{\lambda}\left[\sum_{j=\tau_1}^{\tau_2-1}B_{\lambda}^2(Z_j,Z_{j+1})\right],\\
 \sigma_{10}(\lambda)=\sigma_{01}(\lambda)&:=\frac{1}{E_{\lambda}[\tau_2-\tau_1]}
 E_{\lambda}\left[\Bigl(d(Z_{\tau_2})-d(Z_{\tau_1})\Bigr)\sum_{j=\tau_1}^{\tau_2-1}B_{\lambda}(Z_j,Z_{j+1})\right]\nonumber.
 \end{align}
 Moreover, under $P^{\tt NB}_{\lambda}$ the sequence $\bigl\{(n^{-1/2}(d(Z_n)-n\upsilon_\lambda),n^{-1/2}P_n)\bigr\}_{n\geq1}$ converges weakly to the same two dimensional Gaussian random variable $(X,Y)$.
 \end{prp}
\begin{proof}
We have already proved the first claim. The second claim is immediate from the fact that the distribution
of $$\left(d(Z_{\tau_{2}})-d(Z_{1})-\upsilon_\lambda(\tau_{2}-\tau_1),\sum_{j=\tau_1}^{\tau_{2}-1}B_{\lambda}(Z_j,Z_{j+1})\right)\ \ {\rm under}\ P_{\lambda}$$ is same as that of $$\left(d(Z_{1})-\upsilon_\lambda\tau_1,\sum_{j=0}^{\tau_{1}}B_{\lambda}(Z_j,Z_{j+1})\right)\ \ {\rm under}\ P_{\lambda}^{\tt NB}.$$ 
\end{proof}

We now prove Lemma \ref{lem:replace} by using discussions given in this subsection.
\begin{proof}[Proof of Lemma \ref{lem:replace}.]
By the Markov property, we have that
\begin{align*}
&\Eb\left[E^{\mathbf{T}^*}_{\lambda+h}\left[\left(d(Z_n)-n\upsilon_\lambda\right)\ind_{\{\sigma_{e^*}=\infty\}}\right]\right]\\
&=\mathbb{E}\left[E^{\mathbf{T}^*}_{\lambda}\!\left[\left(d(Z_n)-n\upsilon_\lambda\right)\prod_{i=1}^{n}\frac{A_{\lambda+h}(Z_{i-1},Z_i)}{A_{\lambda}(Z_{i-1},Z_i)}\ind_{\{\sigma_{e^*}>n\}}
\cdot E^{\mathbf{T}^*}_{\lambda+h,Z_n}\left[\ind_{\{\sigma_{e^*}=\infty\}}\right]\right]\right]
\end{align*}
and
\begin{align*}
&\Eb\left[E^{\mathbf{T}^*}_{\lambda}\left[\left(d(Z_n)-n\upsilon_\lambda\right)\prod_{i=1}^{n}\frac{A_{\lambda+h}(Z_{i-1},Z_i)}{A_{\lambda}(Z_{i-1},Z_i)}\cdot\ind_{\{\sigma_{e^*}=\infty\}}\right]\right]\\
&=\mathbb{E}\left[E^{\mathbf{T}^*}_{\lambda}\left[\left(d(Z_n)-n\upsilon_\lambda\right)\prod_{i=1}^{n}\frac{A_{\lambda+h}(Z_{i-1},Z_i)}{A_{\lambda}(Z_{i-1},Z_i)}\ind_{\{\sigma_{e^*}>n\}}
\cdot E^{\mathbf{T}^*}_{\lambda,Z_n}\left[\ind_{\{\sigma_{e^*}=\infty\}}\right]\right]\right].
\end{align*}
Thus, by H\"{o}lder's inequality and Jensen's inequality we have that \eqref{eq:replace} is equal to
\begin{align*}
&\Biggl|\mathbb{E}\Biggl[E^{\mathbf{T}^*}_{\lambda}\Biggl[\left(\dfrac{d(Z_n)-n\upsilon_\lambda}{\sqrt{n}}\right)\ind_{\{\sigma_{e^*}>n\}}\cdot\prod_{i=1}^{n}\frac{A_{\lambda+h}(Z_{i-1},Z_i)}{A_{\lambda}(Z_{i-1},Z_i)}\nonumber\\
&\ \ \ \ \ \ \ \ \ \ \ \ \ \ \ \ \ \ \ \ \ \ \ \ \ \ \ \ \ \ \ \ \ \ \ \ \ \ \ \ \  \  \cdot\left(E^{\mathbf{T}^*}_{\lambda+h,Z_n}\left[\ind_{\{\sigma_{e^*}=\infty\}}\right]-E^{\mathbf{T}^*}_{\lambda,Z_n}\left[\ind_{\{\sigma_{e^*}=\infty\}}\right]\right)\Biggr]\Biggr]\Biggr|
\end{align*}
which is bounded above by $\mathcal{E}^1_{h,n}\cdot\mathcal{E}^2_{h,n}\cdot\mathcal{E}^3_{h,n}$ where
\begin{align*}
\mathcal{E}^1_{n}&:= \Eb\left[E^{\mathbf{T}^*}_{\lambda}\left[\left(\dfrac{d(Z_n)-n\upsilon_\lambda}{\sqrt{n}}\right)^2\ind_{\{\sigma_{e^*}>n\}}\right]\right]^{1/2},\\
\mathcal{E}^2_{h,n}&:=\Eb\left[E^{\mathbf{T}^*}_{\lambda}\left[\left(\prod_{i=1}^{n}\frac{A_{\lambda+h}(Z_{i-1},Z_i)}{A_{\lambda}(Z_{i-1},Z_i)}\right)^4\right]\right]^{1/4},\\
\mathcal{E}^3_{h,n}&:=E_{\lambda}\left[\left(E^{\mathbf{T}^*}_{\lambda+h,Z_n}\left[\ind_{\{\sigma_{e^*}=\infty\}}\right]-E^{\mathbf{T}^*}_{\lambda,Z_n}\left[\ind_{\{\sigma_{e^*}=\infty\}}\right]\right)^{4}\right]^{1/4}.
\end{align*}
It suffices to show that
\begin{align}
\sup_{n}\mathcal{E}^1_{n}&<\infty,\label{eq:e1}\\
\sup_{h,n:h^2n\sim1}\mathcal{E}^2_{h,n}&<\infty,\ {\rm and}\label{eq:e2}\\ 
\lim_{h,n:h^2n\sim1}\mathcal{E}^3_{h,n}&=0.\label{eq:e3}
\end{align}
The estimate \eqref{eq:e1} follows from \eqref{est:ui2}. 

We now prove the estimate \eqref{eq:e2}.  Notice that 
 \begin{align*}
 \left(\prod_{i=1}^n\frac{A_{\lambda+h}(Z_{i-1},Z_i)}{A_{\lambda}(Z_{i-1},Z_i)}\right)^4
 &=\exp(4hP_n-4h^2Q_n+4R_{n,h})
 \intertext{and}
 \prod_{i=1}^n\frac{A_{\lambda+4h}(Z_{i-1},Z_i)}{A_{\lambda}(Z_{i-1},Z_i)}
 &=\exp(4hP_n-16h^2Q_n+R_{n,4h}).
 \end{align*}
 By the estimate \eqref{R_n}, there exists a constant $C_{\lambda}>0$ such that
 $|R_{n,h}|\leq C_{\lambda}$ and $|R_{n,3h}|\leq C_{\lambda}$.
 Since $h^2\sim n^{-1}$, there exists a constant $C'_{\lambda}$ such that $|h^2Q_n|\leq C'_{\lambda}$.
 Thus, there exists a constant $C''_{\lambda}>0$ such that
 \begin{align*}
\left(\prod_{i=1}^n\frac{A_{\lambda+h}(Z_{i-1},Z_i)}{A_{\lambda}(Z_{i-1},Z_i)}\right)^4
\leq C''_{\lambda}\left(\prod_{i=1}^n\frac{A_{\lambda+4h}(Z_{i-1},Z_i)}{A_{\lambda}(Z_{i-1},Z_i)}\right).
\end{align*}
Noticing that
\begin{align*}
E^{\Tr^*}_{\lambda}\left[\prod_{i=1}^n\frac{A_{\lambda+4h}(Z_{i-1},Z_i)}{A_{\lambda}(Z_{i-1},Z_i)}\right]
=E^{\Tr^*}_{\lambda+4h}[1]=1,
\end{align*}
we get the conclusion.\par
Finally, we show \eqref{eq:e3}. Note that 
\begin{equation*}
Q(\lambda):=E^{\mathbf{T}^*}_{\lambda,Z_n}\left[\ind_{\{\sigma_{e^*}=\infty\}}\right]
=P^{\mathbf{T}^*}_{\lambda,Z_n}\left(\sigma_{e^*}=\infty\right)
\end{equation*}
is bounded above by $1$ and monotonically decreasing in $\lambda$. Furthermore, since $d(Z_n)$ converges $P_\lambda$-a.s.\ to $\infty$ as $\nin$, we have that $Q(\tilde{\lambda})$ converges $P_\lambda$-a.s.\ to $1$ for any $\tilde{\lambda}\in(\lambda_c,\mu)$. Fix $t>0$ such that $[\lambda-t,\lambda+t]\subset (\lambda_c,\mu)$ then for $|h|\leq t$, by bounded convergence theorem we then have that
\begin{equation*}
\mathcal{E}^3_{h,n}
\leq \sqrt{2}E_{\lambda}\left[P^{\mathbf{T}^*}_{\lambda-t,Z_n}\left(\sigma_{e^*}=\infty\right)-P^{\mathbf{T}^*}_{\lambda+t,Z_n}\left(\sigma_{e^*}=\infty\right)\right]^{1/4}
\end{equation*}
which converges to $0$ as $\nin$.  
\end{proof}

\subsection{The proof of the differentiability of the speed}
In this subsection, we will prove Theorem \ref{dif}. 

\begin{proof}[Proof of Theorem \ref{dif}]
By Proposition \ref{joint} 
it is now sufficient to prove that $\upsilon_{\lambda}'=E^{\tt NB}_{\lambda}[XY].$
By \eqref{AAA}, we have that
\begin{align}\label{BBB}
& E^{\tt NB}_{\lambda}\left[\left(d(Z_n)-n\upsilon_\lambda\right)\prod_{i=1}^{n}\frac{A_{\lambda+h}(Z_{i-1},Z_i)}{A_{\lambda}(Z_{i-1},Z_i)}
\right]\nonumber\\
 &=E^{\tt NB}_{\lambda}\left[\left(d(Z_n)-n\upsilon_\lambda\right)\exp(hP_n-h^2Q_n+R_{n,h})
 \right].
\end{align}
  Therefore, once we justify that we can pass to the limit in \eqref{BBB}, by using Lemma \ref{lem:replace} and Proposition \ref{joint}
 we will get 
 \begin{align}\label{conv}
 &\frac{1}{hn}E^{\tt NB}_{\lambda+h}\left[d(Z_n)-n\upsilon_\lambda 
 \right]\nonumber\\
&\rightarrow E^{\tt NB}_{\lambda}\left[X\exp\left(Y-\frac{1}{2E_{\lambda}[\tau_2-\tau_1]}E_{\lambda}\left[\sum_{j=\tau_1}^{\tau_2-1}B_{\lambda}^2(Z_j,Z_{j+1})\right]\right)
\right] 
\end{align}
where $(X,Y)$ is the two dimensional Gaussian random variable with the covariance matrix $\Sigma_{\lambda}$.
 Notice that we have shown the continuity of the escape probability in Lemma \ref{p:esc-conti}.
 Since it is shown in \eqref{var} that 
 \begin{align*} 
{\rm Var}(Y)=\frac{1}{E^{\tt NB}_{\lambda}[\tau_1]}E^{\tt NB}_{\lambda}\left[\sum_{j=0}^{\tau_1-1}B_{\lambda}^2(Z_j,Z_{j+1})\right],
\end{align*}
 the above convergence and the integration by parts formula for Gaussian laws implies
 \begin{align*} 
 \frac{\upsilon_{\lambda+h}-\upsilon_\lambda}{h}
 \;\rightarrow\; E_{\lambda}^{\tt NB}\left[X\exp\left(Y-\frac{1}{2}{\rm Var}(Y)\right)\right]
 \;=\;E_{\lambda}^{\tt NB}[XY]=E_{\lambda}[XY].
 \end{align*}
 In order to justify the step
 \eqref{conv}, it suffices to show the uniform integrability of 
 \begin{align*}
 \left\{\frac{1}{hn}\Bigl(d(Z_n)-n\upsilon_\lambda\Bigr)\cdot
 \prod_{i=1}^n\frac{A_{\lambda+h}(Z_{i-1},Z_i)}{A_{\lambda}(Z_{i-1},Z_i)}\right\}_{n\geq1}.
 \end{align*}
 under $P_{\lambda}^{\tt NB}$.
  By H\"{o}lder's inequality, we have
 \begin{align*}
 &E_{\lambda}^{\tt NB}\left[\left(\frac{1}{hn}\Bigl(d(Z_n)-n\upsilon_\lambda\Bigr)\cdot
 \prod_{i=1}^n\frac{A_{\lambda+h}(Z_{i-1},Z_i)}{A_{\lambda}(Z_{i-1},Z_i)}\right)^{6/5}\right]\\
 &\leq E_{\lambda}^{\tt NB}\left[\frac{1}{(hn)^2}\Bigl(d(Z_n)-n\upsilon_\lambda\Bigr)^2\right]^{3/5}
 E_{\lambda}^{\tt NB}\left[\left(\prod_{i=1}^n\frac{A_{\lambda+h}(Z_{i-1},Z_i)}{A_{\lambda}(Z_{i-1},Z_i)}\right)^3\right]^{2/5}
 \end{align*}
 In Lemma \ref{ui},
 we have already seen that $E_{\lambda}^{\tt NB}\left[\frac{1}{(hn)^2}\bigl(d(Z_n)-n\upsilon_\lambda\bigr)^2\right]$ is bounded
 in $n$. That 
$E_{\lambda}^{\tt NB}\left[\left(\prod_{i=1}^n\frac{A_{\lambda+h}(Z_{i-1},Z_i)}{A_{\lambda}(Z_{i-1},Z_i)}\right)^3\right]$
is also bounded in $n$ follows from the estimate \eqref{eq:e2}.
 \end{proof}
 
\section{Uniform moment bounds on regeneration times}\label{s:uni}

In this section we study regeneration and return times for biased random walks on supercritical GW-trees whose offspring law has exponential moments and no deaths (i.e.\ $p_0=0$). 
First, we prove that for any $u\in\Nb$ and $[a,b]\subset (0,\mu)$ we have
\begin{equation}\label{e:uMom}
\sup_{\lambda\in[a,b]}E_{\lambda}^{\tt NB}\left[\tau_1^u\right]=\sup_{\lambda\in[a,b]}\Et_\lambda\left[\left(\tau_2-\tau_1\right)^u\right]<\infty.
\end{equation}
This will be used in the proof of Proposition \ref{p:UniMom} where we consider the case with leaves. Following this, we show that the escape probability is continuous in $\lambda$ thus proving Proposition \ref{p:esc-conti}.

Towards proving \eqref{e:uMom}, we note that, since the interval $[a,b]$ is compact, it suffices to show that for any $\lambda\in(0,\mu)$ and $u\in\Nb$ there exists $\varepsilon>0$ such that 
\[\sup_{|h|\leq \varepsilon}E_{\lambda}^{\tt NB}\left[\tau_1^u\right]=\sup_{|h|\leq \varepsilon}\Et_{\lambda+h}\left[\left(\tau_2-\tau_1\right)^u\right]<\infty.\]
For $\lambda<1$ this follows trivially by choosing $\varepsilon<1-\lambda$ and comparing with a biased random walk on $\Zb$ (e.g.\ Lemma 5.1 of \cite{depeze96}). We consider the case $\lambda\geq 1$ and proceed similarly to Proposition 3 in \cite{PZ} in which it is shown that $E_{\lambda}^{\tt NB}[\tau_1^u]=\Et_\lambda\left[\left(\tau_2-\tau_1\right)^u\right]<\infty$ for any $\lambda\in(0,\mu)$ and $u\in\Nb$. 

Our main contribution here is that we show that this bound is uniform in the bias $\lambda$ in compact intervals for which Remark \ref{r:Rayleigh} will play an important role. 
\begin{rem}\label{r:Rayleigh}
By Rayleigh's monotonicity principle we have that for any infinite tree $\Ts$ and any $v\in\Ts$, 
\[P^{\Ts}_{\lambda,v}(\sigma_e=\infty)\]
 is monotonically decreasing in $\lambda$. This follows using the relationship between electrical networks and reversible Markov chains (see \cite{LP} for further detail).
%
\end{rem}
We now show that the speed $\upsilon_\lambda$ is bounded away from $0$ uniformly in $\lambda$ in compact subsets of $[1,\mu)$.
\begin{lem}\label{l:SpBnd}
Suppose $p_0=0$. 
For any $b\in[1,\mu)$ there exists a constant $c_b>0$ such that 
\[\inf_{\lambda\in[1,b]}\upsilon_\lambda\geq c_b.\]
\end{lem}
\begin{proof}
By Theorem 3.1 of \cite{LPP3}, for $\lambda \in (1,\mu)$ we have that $\upsilon_\lambda\geq (1-\lambda^{-1})^3(1-q_\lambda)^2/12$ where $q_\lambda$ is the smallest non-negative solution to $f(1-\lambda^{-1}(1-q_\lambda))=q_\lambda$. It is immediate from this that for any $a>1$ there exists $c_{a,b}>0$ such that 
\[\inf_{\lambda\in[a,b]}\upsilon_\lambda\geq c_{a,b}.\]
It therefore remains to consider $\lambda$ arbitrarily close to $1$. 


Let $\xi$ be a random variable with the offspring distribution. By Theorem 1.1 of \cite{A1} we have that 
\[\upsilon_\lambda=\Eb\left[\frac{(\xi-\lambda)\tilde{p}_\lambda^{(0)}}{\lambda-1+\sum_{i=0}^\xi\tilde{p}^{(i)}_\lambda}\right]\bigg/\Eb\left[\frac{(\xi+\lambda)\tilde{p}_\lambda^{(0)}}{\lambda-1+\sum_{i=0}^\xi\tilde{p}^{(i)}_\lambda}\right]\]
where $\tilde{p}^{(i)}_\lambda$ are independent copies of $\Pt_\lambda^\Tr(\sigma_e=\infty)$ (which are also independent of $\xi$).

Since $\tilde{p}^{(i)}_\lambda$ are independent of $\xi$ we have that, for $\lambda \in[1,3/2]$, 
\begin{align}
&\Eb\left[\frac{(\xi-\lambda)\tilde{p}_\lambda^{(0)}}{\lambda-1+\sum_{i=0}^\xi\tilde{p}^{(i)}_\lambda}\right]\notag\\
& =\sum_{k=1}^\infty \Pb(\xi=k)\Eb\left[\frac{(k-\lambda)\tilde{p}_\lambda^{(0)}}{\lambda-1+\sum_{i=0}^k\tilde{p}^{(i)}_\lambda}\right] \label{e:SpLo}\\
& \geq p_1(1-\lambda)\Eb\left[\frac{\tilde{p}_\lambda^{(0)}}{\lambda-1+\sum_{i=0}^1\tilde{p}^{(i)}_\lambda}\right]+\frac{1}{4}\sum_{k=2}^\infty p_k\Eb\left[\frac{k\tilde{p}_\lambda^{(0)}}{\lambda-1+\sum_{i=0}^k\tilde{p}^{(i)}_\lambda}\right] \notag
\end{align}
since $\xi-\lambda\geq \xi/4$ for $\xi\geq 2$ and $\lambda \leq 3/2$. Similarly,
\begin{align}\label{e:SpUp}
\Eb\left[\frac{(\xi+\lambda)\tilde{p}_\lambda^{(0)}}{\lambda-1+\sum_{i=0}^\xi\tilde{p}^{(i)}_\lambda}\right]
& \leq 2\sum_{k=1}^\infty p_k\Eb\left[\frac{k\tilde{p}_\lambda^{(0)}}{\lambda-1+\sum_{i=0}^k\tilde{p}^{(i)}_\lambda}\right].
\end{align}

By Remark \ref{r:Rayleigh}, for any tree $\tilde{p}^{(i)}_\lambda$ is decreasing in $\lambda$. Moreover, $\Pb(\tilde{p}^{(i)}_{1+\varepsilon}>0)>0$ for any $\varepsilon\in(0,\mu-1)$. It follows that there exists $c>0$ such that for any $k\geq 1$ and $\lambda\in [1,1+\varepsilon]$ for $\varepsilon>0$ suitably small we have that 
\begin{align}
\frac{k}{k+1} 
& \geq \Eb\left[\frac{k\tilde{p}_\lambda^{(0)}}{\lambda-1+\sum_{i=0}^k\tilde{p}^{(i)}_\lambda}\right] \notag\\
& =\frac{k}{k+1}\left(1-\Eb\left[\frac{\lambda-1}{\lambda-1+\sum_{i=0}^k\tilde{p}^{(i)}_\lambda}\right]\right) \notag\\
& \geq \frac{k}{k+1}\left(1-\Eb\left[\frac{\varepsilon}{\varepsilon+\sum_{i=0}^k\tilde{p}^{(i)}_{1+\varepsilon}}\right]\right) \notag\\
& \geq \frac{ck}{k+1}. \label{e:kkp1}
\end{align}
In particular, we can choose $\varepsilon>0$ sufficiently small such that 
\[p_1(\lambda-1)\Eb\left[\frac{\tilde{p}_\lambda^{(0)}}{\lambda-1+\sum_{i=0}^1\tilde{p}^{(i)}_\lambda}\right]\leq \frac{1}{8}\sum_{k=2}^\infty p_k\Eb\left[\frac{k\tilde{p}_\lambda^{(0)}}{\lambda-1+\sum_{i=0}^k\tilde{p}^{(i)}_\lambda}\right]\]
uniformly over $\lambda \in[1,1+\varepsilon]$. Combining this with \eqref{e:SpLo} and \eqref{e:SpUp} we have
\begin{align*}
\upsilon_\lambda & \geq \frac{1}{16} \sum_{k=2}^\infty p_k\Eb\left[\frac{k\tilde{p}_\lambda^{(0)}}{\lambda-1+\sum_{i=0}^k\tilde{p}^{(i)}_\lambda}\right]\bigg/  \sum_{k=1}^\infty p_k\Eb\left[\frac{k\tilde{p}_\lambda^{(0)}}{\lambda-1+\sum_{i=0}^k\tilde{p}^{(i)}_\lambda}\right]
\end{align*}
which is bounded below for $\lambda \in[1,1+\varepsilon]$ for $\varepsilon>0$ suitably small using \eqref{e:kkp1}
\end{proof}

We now use the Girsanov formula \eqref{Gir} to obtain a useful bound relating the laws for different values of $\lambda$. 
Let $\Delta_n:=\inf\{m\geq 0: d(Z_m)=n\}$ be the first time the walk reaches distance $n$ from the root. 
\begin{lem}\label{l:Gir}
For any tree $\Ts$ of height at least $n$, $\lambda\in(0,1]$ and $h\in(0,\lambda)$ we have that 
\[\Pt_{\lambda-h}^{\Ts}(\Delta_n> m,\sigma_e>m)\leq e^{nh}\Pt_{\lambda}^{\Ts}(\Delta_n> m, \sigma_e>m).\]
\end{lem}
\begin{proof}
First note that the function $F((Z_{k})_{k\geq 0})=\ind_{\{\Delta_n> m,\sigma_e>m\}}$ is measurable with respect to $\Fc_m(\Ts)$ therefore, by the Girsanov formula \eqref{Gir} we have that 
\begin{align}\label{e:Gir}
\Pt_{\lambda-h}^{\Ts}(\Delta_n> m,\sigma_e>m)=\Et_{\lambda}^{\Ts}\left[\ind_{\{\Delta_n> m,\sigma_e>m\}}\prod_{i=1}^m\frac{A_{\lambda-h}(Z_{i-1},Z_i)}{A_{\lambda}(Z_{i-1},Z_i)}\right].
\end{align}

For a walk started from the root, every time the walk takes a step back towards the root it crosses an edge that has previously been crossed. In particular, there is a most recent time that edge was crossed and, due to the tree structure, it must have been crossed directed away from the root. 
It follows that, for any path $(z_k)_{k=0}^m$ in $\Ts$, every pair $(z_{i-1},z_i)$ either corresponds to a unique pair $(z_{j-1},z_j)$ using this coupling or belongs to the unique self avoiding path starting from the root and ending at $z_m$. Denote by $\gamma$ this unique path of length $d(z_m)$. 

For a neighbouring pair of vertices $x,y\in\Ts$ it is straightforward to show that 
\[\frac{A_{\lambda-h}(x,y)A_{\lambda-h}(y,x)}{A_{\lambda}(x,y)A_{\lambda}(y,x)}\leq 1\]
for $\lambda\in(0,1]$ and $h\in(0,\lambda)$.  It follows that, 
\begin{align}\label{e:Canc}
\prod_{i=1}^{d(z_m)}\frac{A_{\lambda-h}(z_{i-1},z_i)}{A_{\lambda}(z_{i-1},z_i)}=\prod_{x\in\gamma\setminus\{z_0,z_m\}}\frac{\lambda+\nu(x)}{\lambda-h+\nu(x)} \leq e^{h(d(z_m)-1)}.
\end{align}
Noting that $\{\Delta_n>m\}\subset\{d(Z_m)< n\}$, combining \eqref{e:Gir} and \eqref{e:Canc} completes the proof.
\end{proof}


An important result that we will use in the following proof is that the distance between regenerations have exponential moments. That is, by Lemma 4.2 of \cite{DGPZ} we have that for any $\lambda\in(0,\mu)$ there exists $\theta(\lambda)=:\theta>0$ such that $\Et_\lambda[e^{\theta d(Z_{\tau_1})}]<\infty$. In fact, we require the stronger uniform moment bound Lemma \ref{l:empMom},
whose proof is a straightforward extension of that of Lemma 4.2 in \cite{DGPZ} using Remark \ref{r:Rayleigh} which we omit.
\begin{lem}\label{l:empMom}
Suppose $p_0=0$. 
For any $[a,b]\subset(0,\mu)$ there exists $\theta(a,b)=:\theta>0$ such that
\begin{align}\label{e:empMom}
\sup_{\lambda\in[a,b]}\Et_\lambda[e^{\theta d(Z_{\tau_1})}]<\infty.
\end{align}
\end{lem}

We now proceed to the main result of this section. This follows similarly to Proposition 3 in \cite{PZ} however, we include the proof since the extension to uniformity over $\lambda$ is delicate.
\begin{prp}\label{p:UnMo}
Suppose $p_0=0$, $b\in[1,\mu)$ and that there exists $\beta>1$ such that $\sum_{k\geq 1}p_k\beta^k<\infty$. 
For all $u\in\Nb$ and $\lambda\in[1,\mu)$ there exists $\varepsilon>0$ such that 
\[\sup_{\lambda \in[1-\varepsilon,b]}\Et_{\lambda}\left[\left(\tau_2-\tau_1\right)^u\right]<\infty.\]
\end{prp}
\begin{proof}
First note that 
\[\Et_{\lambda}[(\tau_2-\tau_1)^u]=E_{\lambda}^{\tt NB}[\tau_1^u]
=\frac{\Eb\left[E^{\Tr^*}_{\lambda}[\tau_1^u\ind_{\{\sigma_{e^*}=\infty\}}]\right]}{\Eb\left[P^{\Tr^*}_{\lambda}(\sigma_{e^*}=\infty)\right]}.\]
Since the denominator $\Eb\left[P^{\Tr^*}_{\lambda}(\sigma_{e^*}=\infty)\right]$ is monotonic in $\lambda$ by
Remark \ref{r:Rayleigh},
it suffices to consider $\sup_{\lambda \in[1-\varepsilon,b]}\Eb\left[E^{\Tr^*}_{\lambda}\left[\tau_1^u\ind_{\{\sigma_{e^*}=\infty\}}\right]\right]$.
Using the uniform exponential moment bound \eqref{e:empMom}, the Cauchy-Schwarz inequality and integration by parts we have
\begin{align}\label{est:long}
& \Eb\left[E^{\Tr^*}_{\lambda}\left[\tau_1^u\ind_{\{\sigma_{e^*}=\infty\}}\right]\right]\\
& = \sum_{n=1}^\infty \Eb\left[E^{\Tr^*}_{\lambda}\left[\tau_1^u;\sigma_{e^*}=\infty,d(Z_{\tau_1})=n\right]\right] \notag\\
& = \sum_{n=1}^\infty \Eb\left[E^{\Tr^*}_{\lambda}\left[\Delta_n^u;\sigma_{e^*}=\infty,d(Z_{\tau_1})=n\right]\right] \notag\\
& \leq \sum_{n=1}^\infty \Eb\left[E^{\Tr^*}_{\lambda}\left[\Delta_n^{2u};\sigma_{e^*}=\infty\right]\right]^{1/2}\Eb\left[P^{\Tr^*}_{\lambda}(d(Z_{\tau_1})=n)\right]^{1/2} \notag\\
& \leq \Eb\left[E^{\Tr^*}_{\lambda}\left[e^{\theta d(Z_{\tau_1})}\right]\right]\sum_{n=1}^\infty e^{-\theta n} \Eb\left[E^{\Tr^*}_{\lambda}\left[\Delta_n^{2u};\sigma_{e^*}=\infty\right]\right]^{1/2} \notag\\
& \leq \Eb\!\left[E^{\Tr^*}_{\lambda}\!\left[e^{\theta d(Z_{\tau_1})}\right]\right]\!\sum_{n=1}^\infty e^{-\theta n}n^{10u}\!
\left(\sum_{k=0}^\infty (k+1)^{2u} \Eb\!\left[P^{\Tr^*}_{\lambda}\!(\Delta_n>kn^{10}, \sigma_{e^*}=\infty)\right]\!\right)^{\!1/2}\!. \notag
\end{align}
By Lemma \ref{l:Gir} we have that for $\varepsilon>0$ suitably small
\begin{align*}
\sup_{h\in(0,\varepsilon)}\Eb\left[P^{\Tr^*}_{1-h}(\Delta_n>kn^{10}, \sigma_{e^*}=\infty)\right]
& \leq\sup_{h\in(0,\varepsilon)}\Eb\left[P^{\Tr^*}_{1-h}(\Delta_n>kn^{10}, \sigma_{e^*}>kn^{10})\right]\\
& \leq e^{\varepsilon n}\Eb\left[P^{\Tr^*}_{1}(\Delta_n>kn^{10}, \sigma_{e^*}>kn^{10})\right].
\end{align*}
Choosing $\varepsilon<\theta/2$ and using \eqref{e:empMom}, it suffices to show that 
\begin{align}\label{est:suff}
\sup_{\lambda\in[1,b]}\sum_{n=1}^\infty e^{-\theta n/2}n^{10u}\left(\sum_{k=0}^\infty (k+1)^{2u} \Eb\left[P^{\Tr^*}_{\lambda}(\Delta_n>kn^{10}, \sigma_{e^*}>kn^{10})\right]\right)^{1/2}<\infty.
\end{align}
 For $k\geq 1$, let
\[\Ac_{1,k,n}:=\bigcup_{m\leq kn^{10}}\{|\nu(Z_m)|\geq \log(kn^{10})^2\}\]
be the event that the walk visits a vertex with at least $\log(kn^{10})^2$ offspring by time $kn^{10}$. By the exponential moments assumption we have that for all $n$ large 
\[
\Eb\left[P_{\lambda}^{\Tr^*}(\Ac_{1,k,n})\right]\leq kn^{10}\Pb(|\nu(e)|\geq \log(kn^{10})^2)\leq e^{-c\log(n^{10})^2}e^{-c\log(k)^2}\]
for some constant $c$ depending only on $\beta$.

Let $N_{k,n}:=|\{m\leq kn^{10}: Z_l\neq Z_m \forall l<m\}|$ be the number of distinct vertices visited by time $kn^{10}$. Set
\[\Ac_{2,k,n}:=\left\{N_{k,n}<\sqrt{kn^{10}}\right\}\cap\left\{\sigma_e>kn^{10}\right\}\]
to be the event that, up to time $kn^{10}$, the walk visits at most $(kn^{10})^{1/2}$ distinct vertices and does not return to the root $e^*$.  On the event $\Ac_{2,k,n}\cap \Ac_{1,k,n}^c$ there is a time $m\leq kn^{10}$ and a vertex $v$ with degree at most $\log(kn^{10})^2$ such that $Z_m=v$ and $v$ is subsequently visited at least $(kn^{10})^{1/2}$ times without a visit to the root. By the Gambler's ruin, for a walk started at $v$ of distance at most $n$ from the root, the probability that the walk returns to $v$ before reaching the root is at most 
$1-1/(2n\log(kn^{10})^2)$ uniformly in $k, m, v$ and $\lambda\geq 1$. It follows that the probability that $v$ is visited by the the walk $(kn^{10})^{1/2}$ times without a visit to the root is at most 
\[\left(1-\frac{1}{2n\log(kn^{10})^2}\right)^{\sqrt{kn^{10}}}.\]

It follows that for $n$ suitably large (independently of $k\geq 1$)
\begin{align*}
\Eb\left[P^{\Tr^*}_\lambda(\Ac_{2,k,n})\right] 
& \leq \Eb\left[P^{\Tr^*}_\lambda(\Ac_{1,k,n})\right] +kn^{10}\left(1-\frac{1}{2n\log(kn^{10})^2}\right)^{\sqrt{kn^{10}}} \\
& \leq 2e^{-c\log(n^{10})^2}e^{-c\log(k)^2}. 
\end{align*}

On the event $\Ac_{2,k,n}^c\cap\{\sigma_e>kn^{10}\}$ there are at least $k^{1/2}n^3$ vertices which are visited by the walk before time $kn^{10}$ with at least time $n^2$ between the first hitting times. Write $\psi_1:=\min\{m>0:Z_l\neq Z_m \forall l<m\}$ and, for $i\geq 2$, 
\[\psi_i:=\min\{m>\psi_{i-1}+n^2:Z_l\neq Z_m \forall l<m\}.\] 
Then, let 
\[\Gc_j=\bigcap_{i=1}^j\left\{\max_{m\leq n^2}|d(Z_{\psi_i})-d(Z_{\psi_i+m})|<n\right\}.\]
We have that
\begin{align}\label{e:awy}
&\Eb\left[P^{\Tr^*}_\lambda(\Delta_n>kn^{10}, \sigma_e>kn^{10},\Ac_{2,k,n}^c)\right] \\
& \leq \Eb\left[P^{\Tr^*}_\lambda\left(\bigcap_{i=1}^{k^{1/2}n^3}\left\{\max_{m\leq n^2}|d(Z_{\psi_i})-d(Z_{\psi_i+m})|<n\right\}\right)\right]\notag \\
& = \prod_{i=1}^{k^{1/2}n^3} \Eb\left[P^{\Tr^*}_\lambda\left(\max_{m\leq n^2}|d(Z_{\psi_i})-d(Z_{\psi_i+m})|<n\big|\Gc_{i-1}\right)\right] \notag\\
&= \prod_{i=1}^{k^{1/2}n^3} \left(1-\Eb\left[P^{\Tr^*}_\lambda\left(\max_{m\leq n^2}|d(Z_{\psi_i})-d(Z_{\psi_i+m})|\geq n\big|\Gc_{i-1}\right)\right]\right).\notag
\end{align}
If the walk regenerates at time $\psi_i$ then $(Z_m)_{m\geq \psi_i}$ is independent of $\Gc_{i-1}$ (conditionally on $Z_{\psi_i}$) therefore \eqref{e:awy} is bounded above by
\begin{align*}
& \prod_{i=1}^{k^{1/2}n^3} \left(1-\Eb\left[P^{\Tr^*}_\lambda\left(\max_{m\leq n^2}|d(Z_{\psi_i})-d(Z_{\psi_i+m})|\geq n, d(Z_m)\geq d(Z_{\psi_i}) \forall m\geq \psi_i\right)\right]\right) \\
&= \prod_{i=1}^{k^{1/2}n^3} \left(1-\Eb\left[P^{\Tr^*}_\lambda\left(\Delta_n<n^2, \sigma_{e^*}=\infty\right)\right]\right) \\
& = \prod_{i=1}^{k^{1/2}n^3} \left(1-\Eb\left[P^{\Tr^*}_\lambda\left(\sigma_e=\infty\right)\right]P_{\lambda}^{\tt NB}\left(\Delta_n<n^2\right)\right).
\end{align*}

We have seen that $\Pt_\lambda\left(\sigma_e=\infty\right)$ is bounded away from $0$ for $\lambda \in[1,b]$ therefore it remains to show that, for $n$ large, 
$P_{\lambda}^{\tt NB}\left(\Delta_n<n^2\right)$ is bounded away from $0$ uniformly in $\lambda\in[1,b]$. By Markov's inequality
\begin{align*}
P_{\lambda}^{\tt NB}\left(\Delta_n\geq n^2\right)
 \leq \frac{E^{\tt NB}_\lambda[\Delta_n]}{n^2}
 \leq \frac{E^{\tt NB}_\lambda[\tau_1]}{n}
 \leq \frac{E^{\tt NB}_\lambda[d(Z_{\tau_1})]}{\upsilon_\lambda n}
\end{align*}
where we have used that there are at most $n$ regenerations up to level $n$ and the formula of the speed \eqref{speed}. By Lemmas \ref{l:SpBnd} and \ref{l:empMom} we then have that this converges to $0$ (uniformly in $\lambda$) as $\nin$ which completes the proof.
\end{proof}

We now prove the following lemma which claims stronger estimates than Proposition \ref{est:sigma} under the assumption that $p_0=0$.
\begin{lem}\label{lem:est:sigma}
Suppose that $p_0=0$ and that $\sum_{k\geq1}p_k\beta^k<\infty$ for some $\beta>1$.
Then for any $\lambda\in(0,\mu)$ and any $\in\mathbb{N}$, we have
\[\Eb\left[E^{\mathbf{T}^*}_{\lambda}\left[(\sigma_{e^*})^l\ind_{\{\sigma_{e^*}<\infty\}}\right]\right]<\infty.\]
\end{lem}
\begin{proof}
On the event $\{\sigma_{e^*}<\infty\}$, we obviously have that 
\begin{align}\label{est:sig-tau}
\sigma_{e^*}\leq\tau_1\ \ P_{\lambda}\mathchar`-a.s.
\end{align}
The estimate \eqref{est:sig-tau} together with Lemma 5.1 in \cite{depeze96} implies the result for $0<\lambda<1$.
 The case $\lambda=1$ can be shown by using \eqref{est:sig-tau} and Theorem 2 in \cite{piau}.
 
 We will show the claim for $1<\lambda<\mu$. Notice that $\sigma_{e^*}=\sigma_{e^*}\wedge \tau_1$ almost surely on the event $\{\sigma_{e^*}<\infty\}$. Therefore similarly to \eqref{est:long}, we obtain that
 \begin{align*}
&\Eb\left[E^{\mathbf{T}^*}_{\lambda}\left[(\sigma_{e^*})^l\ind_{\{\sigma_{e^*}<\infty\}}\right]\right]\notag\\
& = \sum_{n=1}^\infty \Eb\left[E^{\mathbf{T}^*}_{\lambda}\left[(\sigma_{e^*}\wedge \tau_1)^l;\ \sigma_{e^*}<\infty,d(Z_{\tau_1})=n\right]\right] \notag\\
& \leq \Eb\!\left[E^{\Tr^*}_{\lambda}\!\left[e^{\theta d(Z_{\tau_1})}\right]\right]\\
& \qquad\cdot\sum_{n=1}^\infty e^{-\theta n}n^{10k}\left(\sum_{k=0}^\infty (k+1)^{2l} \mathbb{E}\!\left[P^{\mathbf{T}^*}_{\lambda}\!\left(kn^{10}<\sigma_{e^*}<\infty,\ kn^{10}<\Delta_n\right)\right]
\!\right)^{1/2}\!.
\end{align*}
This follows similarly to \eqref{est:suff}.
\end{proof}
We conclude this section by proving Proposition \ref{p:esc-conti}. For this, we first show the following lemma.
\begin{lem}\label{l:Betn}Suppose that $p_0=0$.
For any $[a,b]\subset (0,\mu)$
\[\lim_{n\rightarrow\infty}\sup_{\lambda\in[a,b]}\Eb\left[\Pt^{\Tr^*}_\lambda(n<\sigma_{e^*}<\infty)\right]=0.\]
\end{lem}
\begin{proof}

Let $s_n\rightarrow \infty$ be an increasing sequence that we shall specify later and $\theta>0$ be as in Lemma \ref{l:empMom}. We now split into the cases where $d(Z_{\tau_1})<s_n$ and $d(Z_{\tau_1})\geq s_n$. First, for $d(Z_{\tau_1})\geq s_n$, by Markov's inequality
\begin{flalign*}
&\lim_{n\rightarrow\infty}\sup_{\lambda\in[a,b]}\Eb\left[\Pt^{\Tr^*}_\lambda(n<\sigma_{e^*}<\infty,d(Z_{\tau_1})\geq s_n)\right]\\
& \leq \lim_{n\rightarrow\infty}\sup_{\lambda\in[a,b]}\Eb\left[\Pt^{\Tr^*}_\lambda(d(Z_{\tau_1})\geq s_n)\right] \\
& = \lim_{n\rightarrow\infty}\sup_{\lambda\in[a,b]}\Pt_\lambda(d(Z_{\tau_1})\geq s_n) \\
& \leq \lim_{n\rightarrow\infty}\sup_{\lambda\in[a,b]}\Et_\lambda[e^{\theta d(Z_{\tau_1})}]e^{-\theta s_n} 
\end{flalign*}
which converges to $0$ as $n\rightarrow \infty$ by Lemma \ref{l:empMom}.

For $d(Z_{\tau_1})<s_n$ we note that 
\[\{n<\sigma_{e^*}<\infty, \; d(Z_{\tau_1})<s_n\}\subset\{\max_{m\leq n}d(Z_m)<s_n\}\]
since once the walk reaches level $s_n$ it cannot return to $e^*$ on the event $\{d(Z_{\tau_1})<s_n\}$. With a slight abuse of notation, let $\Zb^+$ denote the tree which is isomorphic to $\Zb^+$. By comparison with a simple random walk on $\Zb^+$ we have that
\begin{flalign*}
&\lim_{n\rightarrow\infty}\sup_{\lambda\in[a,b]}\Eb\left[\Pt^{\Tr^*}_\lambda(n<\sigma_{e^*}<\infty,d(Z_{\tau_1})< s_n)\right]\\
& \leq \lim_{n\rightarrow\infty}\sup_{\lambda\in[a,b]}\Pt_\lambda\left(\max_{m\leq n}d(Z_m)<s_n\right) \\
& \leq \lim_{n\rightarrow\infty}\Pt_\mu^{\Zb^+}\left(\max_{m\leq n}d(Z_m)<s_n\right) \\
& = \lim_{n\rightarrow\infty}\Pt_\mu^{\Zb^+}\left(\kappa(s_n)>n\right)
\end{flalign*}
where $\kappa(l):=\inf\{n\geq 0: d(Z_n)=l\}$ is the first hitting time of level $l$.

A simple calculation shows that $\Et_\mu^{\Zb^+}[\kappa(l)]\leq C_\mu \mu^{l}$. Therefore, choosing $s_n=\log(n)/\log(\mu^2)$ and using Markov's inequality we have that $\Pt_\mu^{\Zb^+}\left(\kappa(s_n)>n\right)\leq C_\mu n^{-1/2}$ and therefore
\[\lim_{n\rightarrow\infty}\sup_{\lambda\in[a,b]}\Eb\left[\Pt^{\Tr^*}_\lambda(n<\sigma_{e^*}<\infty,d(Z_{\tau_1})< s_n)\right]=0.\]
\end{proof}

We are now ready to prove that the escape probability $\Eb[P_{\lambda}^{\Tr^*}(\sigma_{e^*(\Tr)}<\infty)]$ is a continuous function of the bias.
\begin{proof}[Proof of Proposition \ref{p:esc-conti}]
First note that we can assume $p_0=0$ without loss of generality since $P_{\lambda}^{\Tr^*}\left(\sigma_{e^*(\Tr)}=\infty\right)=P_{\lambda}^{\Tr^*_g}\left(\sigma_{e^*(\Tr_g)}=\infty\right)$ $\Pb$-a.s.\ where we recall that $\Tr^*_g$ is the backbone of $\Tr^*$.

Note that for any tree $\Ts$ and any $n\in\mathbb{N}$, the function
$\lambda\mapsto P_{\lambda}^{\Ts}(\sigma_{e^*}<n)$
is continuous since $P_{\lambda}^{\Ts}(\sigma_{e^*}<n)$ only depends on the first $n$ steps of $(Z_n)$.
We now show that for any $n\in\Nb$,
\begin{align}\label{eq:conti-n}
 {\rm the\ function}\ \lambda\mapsto \Eb\left[P_{\lambda}^{\Tr^*}(\sigma_{e^*}<n)\right]\ {\rm is\ continuous.} 
 \end{align}
 For a tree $\Ts$ and $n\in\Nb$ we write $\Ts[n]$ for the truncated tree up to $n\th$ generation, and define 
 \[\mathcal{B}_{n,m}:=\bigcap_{k=1}^n\{\Ts;\ {\rm every\ vertex\ of\ \Ts[{\it n}]\ in\ {\it k}\th\ generation\ has\ at\ most\ {\it m}\ children}\}.\]
 Noticing that $P_{\lambda}^{\Ts}(\sigma_{e^*}<n)$ only depends on $\Ts[n]$ and that
 $\{\Ts[n];\ \Ts\in\mathcal{B}_{n,m}\}$ is a finite set,
 we obtain that the function $\lambda\mapsto \Eb\left[P_{\lambda}^{\Tr^*}\left(\sigma_{e^*}<n\right)\ind_{\mathcal{B}_{n,m}}\right]$ is a continuous function.
 Now the claim \eqref{eq:conti-n} follows since $\mathbb{P}(\mathcal{B}_{n,m}^c)$ is independent of $\lambda$ and 
 converges to $0$ as $m\to\infty$ for any $n$.

In order to deduce the conclusion from \eqref{eq:conti-n}, it suffices to prove that 
$$\Eb\left[P_{\lambda}^{\Tr^*}\left(n<\sigma_{e^*}<\infty\right)\right]$$
 is uniformly convergent to $0$ in $(\lambda_c,\mu)$ as $n\to\infty$. This immediately follows from
 Lemma \ref{l:Betn}.
\end{proof}

\section{Moments of generation sizes of Galton-Watson trees}\label{s:mom}
In this section we prove several technical estimates for subcritical GW-trees which we will require later when showing moment bounds for the time between regenerations of the walk. For this section, we take $W_n$ to be a GW-process with mean number of offspring $\mu:=\Eb[W_1]<1$ which will typically be applied as $\lambda_c$ in Section \ref{s:reg}.

The following lemma gives a bound on the moments of generation sizes of GW-processes.  The main purpose of this lemma is to prove Lemma \ref{l:GWProd}.
\begin{lem}\label{l:GWMom}
Suppose $W_n$ is a GW-process with mean number of offspring $\mu:=\Eb[W_1]<1$ and which satisfies $\Eb[\beta^{W_1}]<\infty$ for some $\beta>1$. Then, for any $m\in\Nb$ there exists $C_m<\infty$ such that $\Eb[W_n^m]\leq C_m\mu^n$.
\end{lem}
\begin{proof}
We prove this inductively in $m$. The case $m=1$ holds with $\Eb[W_n]=\mu^n$ (cf.\ Chapter I.2 of \cite{atne04}). Suppose that for some $m\geq 2$ there exist $C_j<\infty$ for $j=1,...,m-1$ such that $\Eb[W_n^j]\leq C_j\mu^n$. 

Let $W_n^{(1)}, W_n^{(2)},...$  be independent copies of $W_n$ then, using the branching property,
\begin{flalign*}
\Eb[W_{n+1}^m] 
= \Eb[\Eb[W_{n+1}^m|W_1]] 
= \Eb\left[\Eb\left[\left(\sum_{k=1}^{W_1}W_n^{(k)}\right)^m\big|W_1\right]\right].
\end{flalign*}
For $l,m,N\in\Nb$ let 
$\Ic_l^m(N):=\{{\Kr}=(k_1,...,k_m)\in\{1,...,N\}^m:\sum_{j=1}^N\ind_{\bigcup_{i=1}^m\{k_i=j\}}=l\}$ be the $m$-tuples of positive integers at most $N$ with exactly $l$ distinct values. Expanding the term in the above expression and using that $W_n^{(k)}$ are independent of $W_1$ we have that
\begin{flalign}\label{e:lsum}
\Eb[W_{n+1}^m]  
&= \Eb\left[\Eb\left[\sum_{\Kr\in\{1,...,W_1\}^m}\prod_{i=1}^{m}W_n^{(k_i)}\big|W_1\right]\right] \notag \\
&= \sum_{l=1}^m\Eb\left[\sum_{\Kr\in\Ic_l^m(W_1)}\Eb\left[\prod_{i=1}^{m}W_n^{(k_i)}\right]\right].
\end{flalign}

If $\Kr\in\Ic_1^m(W_1)$ then $k_i=k_1$ for all $i$ and, since there are $W_1$ choices of $k_1$, we have
\begin{flalign}\label{e:le1}
\Eb\left[\sum_{\Kr\in\Ic_1^m(W_1)}\Eb\left[\prod_{i=1}^{m}W_n^{(k_i)}\right]\right]
= \Eb\left[\sum_{\Kr\in\Ic_1^m(W_1)}\Eb[(W_n^{(k_1)})^m]\right]
= \Eb[W_1]\Eb[W_n^m].
\end{flalign}
Otherwise, using independence of $W_n^{(1)}, W_n^{(2)},...$ and our induction hypothesis that $\Eb[W_n^j]\leq C_j\mu^n$ for $j\leq m-1$, for $k\in\Ic_l^m(W_1)$ we have
\begin{flalign*}
\Eb\left[\prod_{i=1}^{m}W_n^{(k_i)}\right]\leq \mu^{nl}\left(\max_{j\leq m-1}C_j\right)^l.
\end{flalign*}
There are $\prod_{j=0}^{l-1}(W_1-j)$ choices for the $l$ distinct values in $\{1,...,W_1\}$ then $l^{m-l}$ choices for the remaining $m-l$ duplicates and at most $m!$ orderings of the indices. In particular, for $l\geq 2$,
\begin{flalign}\label{e:lg2}
\Eb\left[\sum_{\Kr\in\Ic_l^m(W_1)}\Eb\left[\prod_{i=1}^{m}W_n^{(k_i)}\right]\right]
\leq \mu^{nl}(\max_{j\leq m-1}C_j)^ll^{m-l}m!\Eb\left[\prod_{j=0}^{l-1}(W_1-j)\right].
\end{flalign}

By the exponential moment assumption we have that $\Eb[\prod_{j=0}^{l-1}(W_1-j)]\leq \Eb[W_1^l]<\infty$. Combining \eqref{e:lsum}, \eqref{e:le1} and \eqref{e:lg2}, we can choose constants $M_l^m$ such that
\begin{flalign*}
\Eb[W_{n+1}^m]
&\leq \mu\Eb[W_n^m]+\sum_{l=2}^mM_l^m\mu^{nl} \\
&\leq \mu^2\Eb[W_{n-1}^m]+\mu\sum_{l=2}^mM_l^m\mu^{(n-1)l}+\sum_{l=2}^mM_l^m\mu^{nl} \\
&= \mu^2\Eb[W_{n-1}^m]+\sum_{l=2}^mM_l^m\mu^{nl}(1+\mu^{-(l-1)}). 
\end{flalign*}
Iterating and using the geometric sum formula yields
\begin{flalign*}
\Eb[W_{n+1}^m]
&\leq \mu^n\Eb[W_1^m]+\sum_{l=2}^mM_l^m\mu^{nl}\sum_{k=0}^{n-1}\mu^{-k(l-1)}\\
&\leq \mu^n\Eb[W_1^m]+\sum_{l=2}^mM_l^m\mu^{nl}\frac{\mu^{-n(l-1)}-1}{\mu^{-(l-1)}-1}
\end{flalign*}
which is bounded above by $C_m\mu^{n+1}$ as required since $\Eb[W_1^m]<\infty$ by the exponential moments assumption.
\end{proof}
The corresponding lower bound holds trivially by noting that $\Pb(W_n\geq 1)\leq\Eb[W_n^m]$ for any $m\in\Nb$ and using that $\Pb(W_n\geq 1)\mu^{-n}$ is decreasing and converges (e.g.\ Theorem B in \cite{LPP1}). This shows that, up to constants, this is the best possible bound.
 
The following result extends Lemma \ref{l:GWMom} to the expectation of products of the generation sizes at varying times. This is an extension of Lemma 2.4.1 in \cite{bo17} which proves this for $m\leq 3$.
\begin{lem}\label{l:GWProd}
Suppose $W_n$ is a GW-process with mean number of offspring $\mu:=\Eb[W_1]<1$ and which satisfies $\Eb[\beta^{W_1}]<\infty$ for some $\beta>1$. Then, for any $m\in\Nb$ there exists $\tilde{C}_m<\infty$ such that for any $(n_i)_{i=1}^m\in\Nb^m$ we have 
\[\Eb\left[\prod_{i=1}^mW_{n_i}\right]\leq \tilde{C}_m\mu^{\max_{l\leq m}n_l}.\]
\end{lem}
\begin{proof}
Let $W_n^{(k)}$ be independent GW-processes for $k\geq 1$. Using the branching property of GW-processes and convexity of polynomials of degree $l\in\Nb$ we have
\begin{flalign}\label{e:GWCond}
\Eb[W_n^l|W_0=j] 
\; =\; \Eb\left[\left(\sum_{k=1}^jW_n^{(k)}\right)^l\right] 
\;\leq\; j^l\Eb\left[\sum_{k=1}^j\frac{(W_n^{(k)})^l}{j}\right] 
\;=\; j^l\Eb[W_n^l].
\end{flalign} 

Without loss of generality let $n_1\leq n_2\leq...\leq n_m$ be ordered. Noting that $W_n$ is a Markov process, by \eqref{e:GWCond} we have
\begin{flalign*}
\Eb\!\left[\prod_{i=1}^mW_{n_i}\right] 
= \Eb\!\left[\Eb\left[W_{n_m}\big|W_{n_{m-1}}\right]\prod_{i=1}^{m-1}W_{n_i}\right] 
\leq \Eb[W_{n_m-n_{m-1}}]\Eb\!\left[W_{n_{m-1}}^2\prod_{i=1}^{m-2}W_{n_i}\right].
\end{flalign*}
Iterating and applying Lemma \ref{l:GWMom} then gives 
\begin{flalign*}
\Eb\left[\prod_{i=1}^mW_{n_i}\right] 
\;\leq\; \prod_{i=1}^m\Eb[W_{n_i-n_{i-1}}^{m+1-i}]
\;\leq\; \prod_{i=1}^mC_{m+1-i}\mu^{n_k-n_{i-1}}
\;\leq\; \tilde{C}_m\mu^{n_m}
\end{flalign*}
where $\tilde{C}_m=(\max_{l\leq m}C_l)^m<\infty$.
\end{proof}

\section{The proof of Proposition \ref{p:UniMom}}\label{s:reg}
The main aim of this section is to prove Corollary \ref{c:2pe} which states that for any closed ball $B$ contained within $(\lambda_c^{1/2},\mu)$ there exists $\varepsilon>0$ such that the time between regenerations has finite $(2+\varepsilon)\th$ moments uniformly over $\lambda\in B$.  We deduce this from the more general result Proposition \ref{p:UniMom}. 

We first state the following lemma which gives a useful bound for the $\alpha\th$ moments of a geometric random variable. This will be used repeatedly throughout this section. 
\begin{lem}\label{l:GeoAlpha}
For any $\alpha>0$ there exists $C_\alpha<\infty$ such that for any $p\in(0,1)$ we have
\[ \sum_{k=1}^\infty k^\alpha p^k(1-p) \leq C_\alpha p(1-p)^{-\alpha}.\]
\end{lem}
\begin{proof}
Note that if $f:\Rb\rightarrow \Rb^+$ is increasing and $g:\Rb\rightarrow \Rb^+$ is decreasing then for $x \in[k,k+1)$ we have that $f(x-1)\leq f(k)\leq f(x)$ and $g(x)\leq g(k)\leq g(x-1)$. Therefore,
\begin{align*}
\sum_{k=1}^\infty f(k)g(k) 
&=\sum_{k=1}^\infty \int_k^{k+1} f(k)g(k)\d x\\
& \leq \sum_{k=1}^\infty \int_k^{k+1} f(x)g(x-1)\d x\\
& =\int_1^\infty f(x)g(x-1)\d x. 
\end{align*}

Take the specific case that $f(x)=x^\alpha$ (which is increasing since $\alpha>0$) and $g(x)=p^x$ (which is decreasing for $p\in(0,1)$). Then, for $p\in[1/2,1)$, we have that
\begin{flalign*}
p^{-1}(1-p)^{1+\alpha}\sum_{k=1}^\infty k^\alpha p^k 
& \leq p^{-2}(1-p)^{1+\alpha} \int_1^\infty x^\alpha p^x\d x \\
& = p^{-2}\left(\frac{1-p}{\log(p^{-1})}\right)^{1+\alpha} \int_{\log(p^{-1})}^\infty x^\alpha e^{-x}\d x \\
& \leq  4\Gamma(1+\alpha) 
\end{flalign*}
since $\left(\frac{1-p}{\log(p^{-1})}\right)^{1+\alpha}\leq 1$. 

For $p\in(0,1/2]$ we have that
\[p^{-1}(1-p)^{1+\alpha}\sum_{k=1}^\infty k^\alpha p^k \leq \sum_{k=1}^\infty k^\alpha 2^{-(k-1)}\]
which converges.
\end{proof}
We now introduce some notation concerning hitting and regeneration times. Recall that $\sigma_x:=\inf\{n\geq 1:Z_n=x\}$ is the first return time to $x\in\Tr$. Let $S(0):=0$, $S(n):=\inf\{k> S(n-1):Z_k,Z_{k-1} \in \Tr_g\}$ for $n\geq 1$ and $Y_n:=Z_{S(n)}$, then $Y_n$ is a $\lambda$-biased random walk on $\Tr_g$ coupled to $Z_n$. Write $\zeta_0:=0$ and for $m=1,2,...$ let 
\[\zeta_m:=\inf\{k>\zeta_{m-1}:d(Y_j)< d(Y_k), d(Y_l)>d(Y_{k-1})  \text{ for all } j<k\leq l\}  \]
be regeneration times for the walk $Y$. We then have that $\tau_k=\inf\{m\geq 0: Z_m=Y_{\zeta_k}\}$ are the corresponding regeneration times for $Z$ and we define $\varrho_k:=Z_{\tau_k}=Y_{\zeta_k}$ to be the regeneration points.
By Proposition 3.4 of \cite{LPP3} we have that there exists, $\Pt_\lambda$-a.s., an infinite sequence of regeneration times $(\tau_k)_{k\geq 1}$ and 
\[\left\{\left(\tau_{k+1}-\tau_k\right),\left(d(\varrho_{k+1})-d(\varrho_k)\right)\right\}_{k\geq1}\]
are \IID (as are the corresponding variables for $Y$).

Let $\xi_f,\xi_g,\xi_h$ be random variables with probability generating functions $f,g$ and $h$ respectively then let $\xi$ be equal in distribution to the number of vertices in the first generation of $\Tr$. Throughout we will assume that $\xi_f$ has some exponential moments. 

\begin{rem}\label{r:mom}
Since the generation sizes of $\Tr_g$ are dominated by those of $\Tr$ we have that $\xi_g$ is stochastically dominated by $\xi$. Using Bayes' law we have that $\Pr(\xi=k)= p_k(1-q^k)(1-q)^{-1} \leq cp_k$ therefore both $\xi$ and $\xi_g$ inherit the exponential moment bounds of $\xi_f$. Furthermore $\Pr(\xi_h=k)= p_kq^k$ therefore $\xi_h$ automatically has exponential moments. 
\end{rem}

We now show that the duration of an excursion in a single trap has finite $\alpha\th$ moments (uniformly for the bias in a small ball). If $p_0=0$ then traps are trivial therefore assume that $p_0>0$. Denote by $\mathbf{T}_h$ a GW tree with this law and $\mathbf{T}_h^*$ the tree $\mathbf{T}_h$ where we append an additional vertex ${e^*}(\mathbf{T}_h)$ as the root in the usual way (for convenience we write ${e^*}$ when there is no confusion). Let $W_k^{\mathbf{T}_h^*}$ denote the $k\th$ generation size of the tree $\mathbf{T}_h^*$. We denote by $\Pt_{\lambda,x}^{\mathbf{T}_h^*}$ the quenched law of the walk with bias $\lambda$ started from $x$. 
\begin{lem}\label{l:subComp}
Suppose $p_0>0$ and $a<1$ then, for any $\alpha<\log(\lambda_c)/\log(a)$,
  \[\sup_{\lambda\geq a}\Eb\left[\Et^{\mathbf{T}_h^*}_{\lambda}\left[\sigma_{e^*}^\alpha\right]\right] <\infty.\]


 \begin{proof}
Write $\underline{\alpha}:=\max\{k\in\Zb: k<\alpha\}$. Throughout we will use that for $N\in\Nb$ and $x_n\in\Rb_+$ for $n=1,...,N$ we have
\begin{flalign}\label{e:Con}
\left(\sum_{n=1}^Nx_n\right)^\alpha\leq N^{\underline{\alpha}}\sum_{n=1}^Nx_n^\alpha
\end{flalign}
which follows from convexity for $\alpha\geq 1$ and the bound $||\cdot||_{1/\alpha}\leq ||\cdot||_1$ for $l^p$ norms with $\alpha<1$. 

We can write
  \[\sigma_{e^*}=\sum_{x \in \mathbf{T}_h^*}v_x \qquad \text{where} \qquad v_x=\sum_{k=0}^{\sigma_{e^*}-1}\ind_{\{Z_k=x\}}\]
  is the number of visits to $x$ before returning to ${e^*}$. By \eqref{e:Con} it then follows that
  \begin{flalign}\label{e:Jen2}
   \Eb\left[\Et^{\mathbf{T}_h^*}_{\lambda}\left[\sigma_{e^*}^\alpha\right]\right] 
   \;=\; \Eb\left[\Et^{\mathbf{T}_h^*}_{\lambda}\left[\left(\sum_{x \in \mathbf{T}_h^*}v_x\right)^\alpha\right]\right]  
    \;\leq\; \Eb\left[d(\mathbf{T}_h^*)^{\underline{\alpha}}\sum_{x \in \mathbf{T}_h^*}\Et^{\mathbf{T}_h^*}_{\lambda}\left[v_x^\alpha\right]\right]
   \end{flalign}
where, using a decomposition up to the first hitting time of $x$ we have that
\[\Et^{\mathbf{T}_h^*}_{\lambda,{e^*}}\left[v_x^\alpha\right]=\Pt^{\mathbf{T}_h^*}_{\lambda,{e^*}}\left(\sigma_x<\sigma_{e^*}\right)\Et^{\mathbf{T}_h^*}_{\lambda,x}\left[v_x^\alpha\right]\leq \Et^{\mathbf{T}_h^*}_{\lambda,x}\left[v_x^\alpha\right].\]

Started from $x$, for the walk to reach to $e^*$ before returning to $x$, the walk must initially move to $\pi(x)$. It follows that the number of visits to $x$ before reaching $e^*$ is geometrically distributed with termination probability 
\begin{flalign}\label{e:GR}
\Pt^{\mathbf{T}_h^*}_{\lambda,x}(\sigma_{e^*}<\sigma_x)=\frac{\lambda}{\lambda+\nu(x)}\cdot \Pt^{\mathbf{T}_h^*}_{\lambda,\pi(x)}(\sigma_{e^*}<\sigma_x)
\end{flalign}
where $\Pt^{\mathbf{T}_h^*}_{\lambda,\pi(x)}(\sigma_{e^*}<\sigma_x)$ depends only on $\lambda$ and the distance between $e^*$ and $x$. By Lemma \ref{l:GeoAlpha} we have that, for some constant $C_{\alpha}$,  
\begin{flalign}\label{e:Drp}
\Et^{\mathbf{T}_h^*}_{\lambda,x}\left[v_x^\alpha\right] \leq C_{\alpha}\Pt^{\mathbf{T}_h^*}_{\lambda,x}(\sigma_{e^*}<\sigma_x)^{-\alpha}\Pt^{\mathbf{T}_h^*}_{\lambda,x}(\sigma_x<\sigma_{e^*})\leq C_{\alpha}\Pt^{\mathbf{T}_h^*}_{\lambda,x}(\sigma_{e^*}<\sigma_x)^{-\alpha}.
\end{flalign}

For $r\in\Nb$ write
 \[\Rc(\lambda,\alpha,r)=\begin{cases} r^\alpha & \text{if } \lambda=1,\\ \lambda^{-r\alpha} & \text{if } \lambda < 1,\\ 1 & \text{if } \lambda > 1,\end{cases}\]
then, by the Gambler's ruin and \eqref{e:GR}, we have that
\[\Pt^{\mathbf{T}_h^*}_{\lambda,x}(\sigma_{e^*}<\sigma_x)^{-\alpha}\leq (1+\lambda^{-1}\nu(x))^\alpha\Rc(\lambda,\alpha,d^*(x))\]
where $d^*(x)$ denotes the distance between $x\in\Tr_h^*$ and the root $e^*$. Substituting this with \eqref{e:Drp} into \eqref{e:Jen2} and using \eqref{e:Con} we have that 
\begin{flalign*}
\Eb\left[\Et^{\mathbf{T}_h^*}_{\lambda,e^*}\left[\sigma_{e^*}^\alpha\right]\right] 
&\leq C_{\alpha}\Eb\left[d(\mathbf{T}_h^*)^{\underline{\alpha}}\sum_{x \in \mathbf{T}_h^*}(1+\lambda^{-1}\nu(x))^{\alpha}\Rc(\lambda,\alpha,d^*(x))\right]\\
&\leq \tilde{C}_{\alpha}\Eb\left[d(\mathbf{T}_h^*)^{\underline{\alpha}}\sum_{x \in \mathbf{T}_h^*}(1+\lambda^{-\alpha}\nu(x)^{\alpha})\Rc(\lambda,\alpha,d^*(x))\right]\\
&\leq \tilde{C}_{\alpha}\Eb\left[d(\mathbf{T}_h^*)^{\underline{\alpha}}\sum_{k=0}^\infty W_k^{\mathbf{T}_h^*}\left(1+\lambda^{-\alpha}\left(W_{k+1}^{\mathbf{T}_h^*}\right)^{\underline{\alpha}+1}\right)\Rc(\lambda,\alpha,k)\right]
\end{flalign*}
where, for the final inequality, we have replaced the sum over vertices in the tree with a sum over the generations and bounded the number of children of a vertex in generation $k$ with the total number of vertices in generation $k+1$.

Since $W_1^{\mathbf{T}_h^*}=1=W_0^{\Tr_h}$ and $W_{k+1}^{\mathbf{T}_h^*}=W_k^{\Tr_h}$ for $k\geq 1$, we have that $1+\lambda^{-\alpha}\left(W_{k+1}^{\mathbf{T}_h^*}\right)^{\underline{\alpha}+1}\leq C\sum_{j=0}^\infty \left(W_j^{\mathbf{T}_h}\right)^{\underline{\alpha}+1}$ for any $k\geq 1$ and a constant $C\leq 1+a^{-\alpha}$. The process $W_n^{\mathbf{T}_h}$ is a GW-process with offspring distribution $\xi_h$ which has mean $\lambda_c$ and exponential moments. It therefore follows from Lemma \ref{l:GWProd} that
\begin{flalign}
\Eb\left[\Et^{\mathbf{T}_h^*}_{\lambda,e^*}\left[\sigma_{e^*}^\alpha\right]\right] 
&\leq C_{\alpha}\Eb\left[\sum_{k_1=0}^\infty\dots\sum_{k_{2\underline{\alpha}+2}=0}^\infty \Rc(\lambda,\alpha,k_1)\prod_{j=1}^{2\underline{\alpha}+2} W_{k_j}^{\mathbf{T}_h^*}\right] \notag\\
&\leq C_{\alpha}\sum_{k_1=0}^\infty\dots\sum_{k_{2\underline{\alpha}+2}=0}^\infty \Rc(\lambda,\alpha,k_1)\Eb\left[\prod_{j=1}^{2\underline{\alpha}+2} W_{k_j}^{\mathbf{T}_h^*}\right] \notag\\
&\leq \tilde{C}_{\alpha}\sum_{k_1=0}^\infty\dots\sum_{k_{2\underline{\alpha}+2}=0}^\infty \Rc(\lambda,\alpha,k_1)\lambda_c^{\max_{j\leq 2\underline{\alpha}+2}k_j}. \label{e:GWSum}
\end{flalign}

Taking first those terms in \eqref{e:GWSum} where $k_1\geq k_j$ for all $j$, we have
\begin{flalign*}
&\sum_{k_1=0}^\infty\dots\sum_{k_{2\underline{\alpha}+2}=0}^\infty \ind_{\{k_1=\max_{j\leq 2\underline{\alpha}+2}k_j\}}\Rc(\lambda,\alpha,k_1)\lambda_c^{k_1} \\
& \leq (2\underline{\alpha}+1)\sum_{k_1=0}^\infty (k_1+1)\Rc(\lambda,\alpha,k_1)\lambda_c^{k_1}
\end{flalign*}
which is bounded above uniformly over $\lambda\geq a$ since $a^{-\alpha}\lambda_c<1$ by our choice of $\alpha$.

Next, writing $m=\max_{j=2,...,2\underline{\alpha}+2}k_j$, taking the remaining terms in \eqref{e:GWSum} and noting that 
\[\sum_{k_1=0}^{m-1}\Rc(\lambda,\alpha,k_1) \leq m^2a^{-m}\]
we have 
\begin{flalign*}
\sum_{k_1=0}^\infty\dots\sum_{k_{2\underline{\alpha}+2}=0}^\infty \ind_{\{k_1<m\}}\Rc(\lambda,\alpha,k_1)\lambda_c^{m} & \leq (2\underline{\alpha}+1)\sum_{m=1}^\infty m^3a^{-m\alpha}\lambda_c^{m}
\end{flalign*}
which is finite by our choice of $\alpha$.
 \end{proof}
\end{lem}

Let $\chi_k:=S(k+1)-S(k)$ denote the total time taken between $Z_n$ making the $k\th$ and $(k+1)\th$ transition along the backbone. This time consists of
\[N_k:=\sum_{n=S(k)+1}^{S(k+1)}\ind_{\{Z_n=Y_k\}}\]
excursions into the finite trees appended to the backbone at this vertex and one additional step to the next backbone vertex. Write $\vartheta_k^{(0)}:=S(k)$ and $\vartheta_k^{(j)}:=\inf\{n>\vartheta_k^{(j-1)}:Z_n=Y_k\}$ for $j\geq 1$ to be the hitting times of the backbone after time $S(k)$. We can then write
\begin{flalign}\label{e:etaK}
\chi_k:=1+\sum_{j=1}^{N_k}\gamma_{k,j} \qquad \text{ where } \qquad \gamma_{k,j}:=\vartheta_k^{(j)}-\vartheta_k^{(j-1)}
\end{flalign}
is the duration of the $j\th$ such excursion. 

\begin{proof}[Proof of Proposition \ref{p:UniMom}]
Recall that $\underline{\alpha}:=\max\{k\in\Zb: k<\alpha\}$ and write $\overline{\alpha}:=\min\{k\in\Zb: k\geq\alpha\}$.
We therefore have that $\Et^{\tt NB}_\lambda\left[\tau_1^\alpha\right] $ can be written as
\begin{flalign*}
\Et^{\tt NB}_\lambda\left[\tau_1^\alpha\right] 
&= \Et^{\tt NB}_\lambda\left[\left(\sum_{k=1}^{\zeta_1}\chi_k\right)^\alpha\right] \\
&\leq \Et^{\tt NB}_\lambda\left[\zeta_1^{\underline{\alpha}}\sum_{k=1}^{\zeta_1}\chi_k^\alpha\right] 
\end{flalign*}
by \eqref{e:Con}. Using \eqref{e:Con} again with the decomposition \eqref{e:etaK} we can write this as 
\begin{flalign*}
  &\Et^{\tt NB}_\lambda\left[\zeta_1^{\underline{\alpha}}\sum_{k=1}^{\zeta_1} \!\left(1+\sum_{j=1}^{N_k}\gamma_{k,j}\right)^\alpha\!\right] \\
 & \qqquad  \leq \Et^{\tt NB}_\lambda\left[\zeta_1^{\underline{\alpha}}\sum_{k=1}^{\zeta_1}(N_k+1)^{\underline{\alpha}}\left(1+\sum_{j=1}^{N_k}\gamma_{k,j}^\alpha\right) \right]. 
\end{flalign*}

The excursion times $\gamma_{k,j}$ are distributed as the first return time to $e^*$ for a walk started from $e^*$ on $\Tr_h^*$. Moreover, under $\Pt^{\tt NB}_\lambda$, they are independent of the backbone, the buds and the walk on the backbone and buds. In particular, they are independent of the regeneration times of $Y$ and the number of excursions therefore the above expectation can be bounded above by 
\begin{flalign*}
 \Eb\left[\Et^{\Tr_h^*}_{\lambda}\left[\sigma_{e^*}^\alpha\right]\right] \Et^{\tt NB}_\lambda\left[\zeta_1^{\underline{\alpha}}\sum_{k=1}^{\zeta_1}(N_k+1)^{\overline{\alpha}} \right].
\end{flalign*}
Where, by Lemma \ref{l:subComp}, we have that $\sup_{\lambda\in[a,b]}\Eb\left[\Et^{\Tr_h^*}_{\lambda}\left[\sigma_{e^*}^\alpha\right]\right] <\infty$.

%

Let $(z_j)_{j=0}^\infty$ denote the ordered distinct vertices visited by $Y$ and 
\[\Lc(z,j):=\sum_{k=0}^j \ind_{\{Y_k=z\}}, \quad \Lc(z):=\Lc(z,\infty)\]
the local times of the vertex $z$. Write 
\[M_{z,l}:=\sum_{j=0}^\infty \ind_{\left\{Z_j=z,\; Z_{j+1} \notin \Tr_g, \; \Lc(z,j)=l\right\}}\]
to be the number of excursions from $z$ (by $Z$) on the $l\th$ visit to $z$ (by $Y$) for $l=1,...,\Lc(z)$ and $\Jc:=|\{Y_j\}_{j=1}^{\zeta_1-1}|$ the number of distinct vertices visited by $Y$ between time $1$ and time $\zeta_1-1$. 
Each $k\leq \zeta_1$ corresponds to a unique pair $(z_j,l)$ with $j\leq \Jc$ and $l\leq \Lc(z_j)$ with $M_{z_j,l}=N_k$ therefore 
\begin{flalign}
& \Et^{\tt NB}_\lambda\left[\zeta_1^{\underline{\alpha}}\sum_{k=1}^{\zeta_1}(N_k+1)^{\overline{\alpha}} \right]  \notag \\
& \quad  = \Et^{\tt NB}_\lambda\left[\zeta_1^{\underline{\alpha}}\sum_{j=1}^{\Jc}\sum_{l=1}^{\Lc(z_j)}(M_{z_j,l}+1)^{\overline{\alpha}} \right] \notag \\
& \quad = \sum_{j=1}^{\infty}\sum_{l=1}^{\infty}\Et^{\tt NB}_\lambda\left[\zeta_1^{\underline{\alpha}}\ind_{\{j\leq \Jc, \; l\leq \Lc(z_j)\}}(M_{z_j,l}+1)^{\overline{\alpha}}\right] \notag \\
& \quad \leq \sum_{j=1}^{\infty}\sum_{l=1}^{\infty}\Big(\Et^{\tt NB}_\lambda\left[\zeta_1^{2\underline{\alpha}}\ind_{\{j\leq \Jc, \; l\leq \Lc(z_j)\}}\right]\Et^{\tt NB}_\lambda\left[(M_{z_j,l}+1)^{2\overline{\alpha}} \right]\Big)^{1/2} \label{e:CSBnd}
\end{flalign}
by the Cauchy-Schwarz inequality. For all $1\leq j\leq \Jc$ we have that $\Lc(z_j)\leq \zeta_1$; moreover, $\Jc\leq \zeta_1$ therefore \[\ind_{\{j\leq \Jc, \; l\leq \Lc(z_j)\}}\leq\ind_{\{j,l\leq \zeta_1\}}.\] 


Due to the independence structure of the GW-tree, for any fixed $j$ the distribution of the number of children of $z_j$ is equal to the distribution of the number of children of the root. 
Since the root does not have a parent, we have that the walk is more likely to take an excursion into one of the neighbouring traps when at the root than from a vertex with the same number of children. 
We can, therefore, stochastically dominate the number of excursions from a backbone vertex by the number of excursions from the root to see that $\Et^{\tt NB}_\lambda\left[(M_{z_j,l}+1)^{2\overline{\alpha}} \right]\leq \Et^{\tt NB}_\lambda\left[(M_{z_0,1}+1)^{2\overline{\alpha}} \right]$. 


Using this and the Cauchy-Schwarz inequality, the expression \eqref{e:CSBnd} is bounded above by
\begin{flalign*}
&  \Et^{\tt NB}_\lambda\!\left[\zeta_1^{4\underline{\alpha}}\right]^{1/4}\Et^{\tt NB}_\lambda\!\left[(M_{z_0,1}+1)^{2\overline{\alpha}} \right]^{1/2}\sum_{j,l=1}^{\infty}\Pt^{\tt NB}_\lambda\left(j,l\leq \zeta_1\right)^{1/4}.
\end{flalign*}
By Remark \ref{r:mom} the offspring distribution $\xi_g$ has exponential moments, we therefore have that the time between regenerations of $Y$ has finite $4\underline{\alpha}$ moments uniformly over $\lambda \in[a,b]$ by Proposition \ref{p:UnMo}. That is, $\sup_{\lambda\in[a,b]}\Et^{\tt NB}_\lambda\left[\zeta_1^{4\underline{\alpha}}\right]<\infty$.

Write $W_n$ and $W^g_n$ to be the GW-processes associated with $\Tr$ and $\Tr_g$. The number of excursions from the root is geometrically distributed with termination probability $1-p_{ex}$ where 
\[p_{ex}:=\frac{W_1-W_1^g}{W_1}.\]
Using Lemma \ref{l:GeoAlpha} we therefore have that, for a constant $C$ independent of $\lambda$,  
\[\Et^{\tt NB}_\lambda\left[(M_{z_0,1}+1)^{2\overline{\alpha}} \right] \;\leq \; C\mathbb{E}[(1-p_{ex})^{-2\overline{\alpha}}] \;\leq \; 
C\mathbb{E}[W_1^{2\overline{\alpha}}]  \;< \;\infty\]
since $W_1\ed \xi$ which has exponential moments. 

It remains to show that 
\begin{flalign}\label{e:dubsum}
\sum_{j=1}^{\infty}\sum_{l=1}^{\infty}\Pt^{\tt NB}_\lambda\left(j,l\leq \zeta_1\right)^{1/4}
\end{flalign}
is finite. Note that $\Pt^{\tt NB}_\lambda\left(j,l\leq\zeta_1\right)=\Pt^{\tt NB}_\lambda\left(\zeta_1\geq l\right)$ whenever $l\geq j$. Using Chebyshev's inequality we can then bound \eqref{e:dubsum} above by
\begin{flalign*}
 2\sum_{j=1}^{\infty}\sum_{l=j}^{\infty}\Pt^{\tt NB}_\lambda\left(\zeta_1\geq l\right)^{1/4} \leq 2\sum_{j=1}^{\infty}\sum_{l=j}^{\infty} \left(\frac{\Et^{\tt NB}_\lambda\left[\zeta_1^u\right]}{l^u}\right)^{1/4}
\end{flalign*}
for any integer $u$. In particular, we have that $\sup_{\lambda\in[a,b]}\Et^{\tt NB}_\lambda\left[\zeta_1^u\right]$ is finite for any integer $u$ by Proposition \ref{p:UnMo}. Choosing $u>8$ we then have that this sum is finite which completes the proof
of the moment estimate of $\tau_1$.
\end{proof}

Finally, we complete the proof of Proposition \ref{est:sigma}.  
\begin{proof}[Proof of Proposition \ref{est:sigma}]
We first observe that
\begin{align}\label{est:last}
\Eb\left[E_{\lambda}^{\mathbf{T}^*}\left[\sigma_{e^*(\mathbf{T})}\ind_{\{\sigma_{e^*(\mathbf{T})}<\infty\}}\right]\right]
\leq E_{\lambda}\left[1+\sum_{k=1}^{\sigma^Y_{e^*(\Tr_g)}}\chi_k\ ;\ \sigma_{e^*(\mathbf{T})}<\infty\right],
\end{align}
where $\sigma^Y_{e^*(\Tr_g)}$ is the first time that $(Y_n)$ returns to $e^*(\Tr_g)$. The reason why \eqref{est:last}
is an inequality is that the walk $(Z_n)$ may enter traps attached to $e(\Tr)=e(\Tr_g)$ and return to $e^{*}(\Tr)=e^*(\Tr_g)$ without any transitions on the backbone $\Tr_g$. It is straightforward to check \eqref{est:last}
by using Lemma \ref{lem:est:sigma}, Lemma \ref{l:subComp} and arguments in the proof of Proposition \ref{est:sigma}.
\end{proof}



\section*{Acknowledgements}
We would like to thank Ryokichi Tanaka for suggesting to apply the methods in \cite{M} for studying the differentiability of the speed of biased random walks on Galton-Watson trees. 
Our gratitude goes to Pierre Mathieu for helpful discussions during the second author's stay in Marseille.  
We would also like to thank Gerard Ben Arous for several discussions about fluctuation-dissipation theorems. 
Our thanks also go out to the anonymous referee for a thorough examination of the preliminary version.
A.\ B.\ acknowledges support of NUS grant R-146-000-260-114.
Y. T. was supported by JSPS KAKENHI Grant Number JP16H06338.

\end{document}